\documentclass[11pt]{article}

\usepackage[mathscr]{eucal}
\usepackage{color}
\usepackage{amsmath}
\usepackage{amssymb}
\usepackage{amsfonts}
\usepackage{amscd}
\usepackage{amsthm}
\usepackage{latexsym}
\usepackage{graphicx}
\usepackage{subfig}

\def\be{\begin{equation}}
\def\ee{\end{equation}}
\def\bse{\begin{subequations}}
\def\ese{\end{subequations}}

\usepackage{latexsym}

\let\er\eqref

\let\be\beta

\newcommand{\R}{{\mathbb R}}

\newtheorem{theorem}{Theorem}
\newtheorem{lemma}[theorem]{Lemma}
\newtheorem{proposition}[theorem]{Proposition}
\newtheorem{remark}[theorem]{Remark}

\newtheorem{definition}[theorem]{Definition}

\def\bse{\begin{subequations}}
\def\ese{\end{subequations}}

\usepackage{geometry}

\title{Keller-Segel model with Logarithmic Interaction and nonlocal reaction term}

\author{ Shen Bian\footnote{Corresponding author. Email: \texttt{bianshen66@163.com}. }
  \quad Quan Wang\footnote{Email: \texttt{2021201044@buct.edu.cn}. } \\
  \
  \\
Beijing University of Chemical Technology, 100029, Beijing. }
\date{}

\begin{document}
\let\cleardoublepage\clearpage

\maketitle

\begin{abstract}
We investigate the global existence and blow-up of solutions to the Keller-Segel model with nonlocal reaction term $u\left(M_0-\int_{\R^2} u dx\right)$ in dimension two. By introducing a transformation in terms of the total mass of the populations to deal with the lack of mass conservation, we exhibit that the qualitative behavior of solutions is decided by a critical value $8\pi$ for the growth parameter $M_0$ and the initial mass $m_0$. For general solutions, if both $m_0$ and $M_0$ are less than $8\pi$, solutions exist globally in time using the energy inequality, whereas there are finite time blow-up solutions for $M_0>8\pi$ (It involves the case $m_0<8\pi$) with any initial data and $M_0<8\pi<m_0$ with small initial second moment. We also show the infinite time blow-up for the critical case $M_0=8 \pi.$ Moreover, in the radial context, we show that if the initial data $u_0(r)<\frac{m_0}{M_0} \frac{8 \lambda}{(r^2+\lambda)^2}$ for some $\lambda>0$, then all the radially symmetric solutions are vanishing in $L_{loc}^1(\R^2)$ as $t \to \infty$. If the initial data $u_0(r)>\frac{m_0}{M_0} \frac{8 \lambda}{(r^2+\lambda)^2}$ for some $\lambda>0$, then there could exist a radially symmetric solution satisfying a mass concentration at the origin as $t \to \infty.$
\end{abstract}
\section{Introduction}
\def\theequation{1.\arabic{equation}}\makeatother
\setcounter{equation}{0}
\def\thetheorem{1.\arabic{theorem}}\makeatother
\setcounter{theorem}{0}

In this work, we study the Keller-Segel model with logistic sources in dimension two
\begin{align}\label{fisherks}
\left\{
      \begin{array}{ll}
      u_t = \Delta u-\nabla \cdot \left(u \nabla c\right)+u \left( M_0-\int_{\R^2} u dx  \right), ~~& x\in \R^{2},~t\geq 0,\\
     -\Delta c =u , ~~ & x\in \R^{2} ,~t\geq 0, \\
     u(x,0) =u_{0}(x) \ge 0, ~~& x\in \R^2.
      \end{array}\right.
\end{align}
Here the initial mass is defined as $m_0:=\int_{\R^2} u_0(x) dx.$ This model is developed to describe the biological phenomenon chemotaxis \cite{KS70,Pat53} by introducing nonlocal terms in the logistic growth factor. The first equation states the random (Brownian) diffusion of the cells with a bias directed by the chemoattractant concentration. The chemoattractant $c$ is directly released by the cells, diffuses on the substrate and can be expressed
\begin{align}
c(x, t) = -\frac{1}{2\pi} \int_{\mathbb R^2} u(y,t) \log |x-y| dy.
\end{align}
The model at hand supposes that the cells production with logistic sources. Here $M_0$ can be viewed as the mass capacity \cite{bp07} and sometimes also called Malthusian parameter, induces an exponential growth for low density populations \cite{NT13}. The nonlocal term $\int_{\R^2} u dx$ describes the influence of the total mass of the species in the growth of the population, which is a competitive term limiting such growth \cite{NT13}.

As pointed out in \cite{DP04,bp07}, in the absence of the reaction term $u \left( M_0-\int_{\R^2} u dx  \right)$, \er{fisherks} has the property that the total mass is conserved $\int_{\R^2} u_0(x) dx=\int_{\R^2}u(x,t)dx$ and the solution of
\begin{align}\label{ks}
\left\{
      \begin{array}{ll}
      u_t = \Delta u-\nabla \cdot \left(u \nabla c\right), ~~& x\in \R^{2},~t\geq 0,\\
     -\Delta c =u , ~~ & x\in \R^{2} ,~t\geq 0, \\
     u(x,0) =u_{0}(x) \ge 0, ~~& x\in \R^2.
      \end{array}\right.
\end{align}
exists globally when the initial mass $m_0<8\pi,$ while the solution becomes a singular measure for $m_0>8\pi.$ A natural problem is that does the nonlocal reaction change this effect?

To analyze this issue, we consider the total mass
\begin{align}
m(t)=\int_{\R^n} u(x,t)dx
\end{align}
which satisfies
\begin{align}
  \frac{dm(t)}{dt}=m(t) \left( M_0-m(t) \right)
\end{align}
and we obtain
\begin{align}\label{mt}
m(t)=\frac{M_0}{1+\overline{C}e^{-M_0 t}}.
\end{align}
where
\begin{align}
\overline{C}=\frac{M_0-m_0}{ m_0}.
\end{align}
We find that the mass increases in time when $m_0<M_0$ and versus when $m_0>M_0,$ in both cases it converges to $M_0$ as $t \to \infty.$
The aim of this paper is to explore the influence of $M_0,m_0$ on the dynamics of solutions when the nonlocal reaction is present.

In order to put our system in perspective, we recall some related cases. Actually, chemotaxis models with local sources (where $u(M_0-\int_{\R^2}udx)$ is replaced by $a_0-u(a_1-a_2 u^\gamma)$ have been studied recently, which describes the situation where the influence of the nonlocal terms is neglected, and the global existence results have been obtained for different parameters $a_0,a_1,a_2$, see \cite{GST16,NT13,TW07,WMZ14} and the references therein. Logistic growth described by nonlocal terms has been used in the context of chemotaxis, which suggests a growth coefficient rate in a competitive system modelling cancer cells behavior. For instance, the nonlocal term in \cite{25} is given in the form $\mu_1 u\left( 1-\int_{\Omega} k_1(x,y) u(y) dy \right)$ where $u$ denotes the cancer cells density. We refer the reader to \cite{NT13,25} for more details.

To illustrate our results, we observe that the striking feature of \er{fisherks} is the absence of mass conservation which makes a difficult task to work with \er{fisherks} directly. To overcome such difficulty, we employ a transformation
\begin{align}\label{tritri}
\rho(x,t)=\frac{u(x,t)}{m(t)}
\end{align}
where $m(t)$ is defined as \er{mt}, and it transforms system \er{fisherks} into a diffusion-aggregation equation for $\rho(x,t)$ as follows
\begin{align}\label{rhorho}
\left\{
      \begin{array}{ll}
      \rho_t = \Delta \rho-m(t) \nabla \cdot \left(\rho \nabla w\right) , ~~& x\in \R^{2},~t\geq 0,\\
     -\Delta w =\rho , ~~ & x\in \R^{2} ,~t\geq 0, \\
     \rho(x,0) =\frac{u_{0}(x)}{m_0} \ge 0, ~~& x\in \R^2.
      \end{array}\right.
\end{align}
Here we should mention that \er{rhorho} admits mass conservation $\int_{\R^2}\rho(x,t)dx=\int_{\R^2} \rho(x,0) dx=1$ via the transformation, the price we pay is that the coefficient in front of the concentration is $m(t)$ instead of a constant which also brings barriers to mathematical analysis. We remark that for the case that the concentration with constant sensitivity $\chi$ (without loss of generality, we impose $\chi=1$), there have been a large amount of results \cite{BCM08,bdp06,bp07} and there is a threshold $8\pi$ on the initial mass separating global existence and finite time blow-up as we mentioned before. Although we can't take advantage of the initial mass to determine the dynamical behaviors of \er{rhorho} due to $\int_{\R^2}\rho(x,0)dx=1,$ this fact still stimulates us to employ the mass conservation of \er{rhorho} to establish global existence and blow-up to system \er{fisherks}. Precisely, the main results of the general solutions are stated as follows in connection with $M_0,m_0$ and $8\pi.$
\begin{itemize}
 \item Global existence: for $m_0<M_0<8\pi,$ there exists a weak solution globally in time with bounded initial second moment. Furthermore, the weak solution satisfies the energy inequality \er{Fu}, see Theorem \ref{globalexist}. In addition, for $M_0<m_0<8\pi,$ all solutions of \er{fisherks} exist globally by the comparison principle.
 \item $m_0<M_0=8\pi:$ solutions exist globally in time, see Theorem \ref{globalexist}. For $M_0=8\pi,$ every stationary solution uniquely assumes a radially symmetric form in $\R^2$ up to translation and has  infinite second moment. This result provides that the solution blows up as a delta dirac at the center of mass as $t \to \infty$ with finite second moment at any time. We will comment further on these issues in Section \ref{sec3}.
 \item Finite time blow-up: for $M_0>8\pi,$ the weak solution blows up at finite time, see Theorem \ref{blowup}. For $M_0<8\pi<m_0,$ if the initial second moment is less than a constant depending on $M_0,m_0,$ then there exist solutions blow up at finite time, see Theorem \ref{blowup2}.
 \item For $M_0<m_0=8\pi,$ we can infer from \cite{BCM08} that solutions might have infinite time blow-up by the comparison principle. However, we can't exclude the possibility that solutions may be global in time. For $M_0=8\pi<m_0$, both global existence and finite time blow-up can occur, although this is still an open question.
\end{itemize}

Let's emphasise that in the radial context, there are steady states to \er{fisherks} only for $M_0=8\pi$ given by a one-parameter family $U_s(r)=\frac{8 \lambda}{(r^2+\lambda)^2}$ with $\lambda>0.$ The stationary solutions play a critical role on the initial data separating global existence and blow-up. If the initial data $u_0(r)<\frac{m_0}{M_0} \frac{8 \lambda}{(r^2+\lambda)^2}$, then all the radially symmetric solutions are vanishing in $L_{loc}^1(\R^2)$ as $t \to \infty$. If the initial data $u_0(r)>\frac{m_0}{M_0} \frac{8 \lambda}{(r^2+\lambda)^2}$, then there could exist a radially symmetric solution satisfying a mass concentration at the origin as $t \to \infty.$  See Section \ref{radialglobal} and Section \ref{radialblowup} for more details.

The results are organized as follows. Section \ref{globalfinite} shows the global existence and the finite time blow-up of the weak solution to \er{fisherks}. Section \ref{sec1} detects the global existence of solutions with the help of the energy inequality for $M_0<8\pi$. The finite time blow-up is considered in Section \ref{sec2} with bounded initial second moment for $M_0>8\pi$. Section \ref{sec3} explores the infinite time blow-up when $M_0=8\pi.$ Section \ref{sec5} is devoted to the global existence and the mass concentration at the origin for the radially symmetric solutions.

\section{Global existence and finite time blow-up for general solutions}\label{globalfinite}
\def\theequation{2.\arabic{equation}}\makeatother
\setcounter{equation}{0}
\def\thetheorem{2.\arabic{theorem}}\makeatother
\setcounter{theorem}{0}

In this section, we prove the global well-posedness for solutions with small $M_0$ and blow-up for large $M_0.$ Here an energy inequality (which is derived to prove global existence) introduces a threshold on $M_0$, below which the solution exists globally and above which the finite time blow-up occurs with the aid of the second moment. In order to analyse more precisely, we use the usual definition of solutions in distribution sense as in \cite{bp07}.
\begin{definition}
Let $u_0(x) \ge 0$ be the initial data satisfying $\int_{\R^2} u_0(x) |\log u_0(x)| dx +\int_{\R^2} |x|^2 u_0(x) dx<\infty$ and $T\in (0,\infty).$ $u$ is a weak solution to system \er{fisherks} if it satisfies:
\begin{itemize}
  \item[(i)] Regularity:
  \begin{align*}
   & u \in L^\infty\left(0,T;L_+^1 \cap L^p(\R^2) \right), ~~\mbox{for~any}~~1 \le p \le \infty,\\
   & \partial_t u \in L^{2} \left(0,T;H^{-1}(\R^2) \right), \\
   & u^r \in L^2\left(0,T;H^1(\R^2)\right)~~\mbox{for~any}~~r \ge 1/2.
  \end{align*}
  \item[(ii)] For any $\psi \in C_0^\infty(\R^2)$ and any $0<t<\infty$,
  \begin{align}
&\int_{\R^2} \psi u(\cdot,t) dx-\int_{\R^2} \psi u_{0}(x)dx \nonumber \\
=&\int_0^T \int_{\R^2} \Delta \psi u dxdt -\frac{1}{4\pi} \iint_{\R^2 \times \R^2} \frac{[\nabla \psi(x)-\nabla \psi(y)]\cdot(x-y)}{|x-y|^2} u(x,t) u(y,t) dxdydt\nonumber \\
&+\int_0^T \int_{\R^2} u(\cdot,t) \psi(x) dx \left(M_0-\int_{\R^2} u(\cdot,t) dx \right) dt.
\end{align}
\end{itemize}
\end{definition}

\subsection{Global existence for $M_0 \le 8\pi$}\label{sec1}
In this subsection, we will show the global existence of a weak solution for both cases $M_0=8\pi$ and $M_0<8\pi$. Firstly we recall the logarithmic HLS inequality which is prepared to the estimates of the weak solution.
\begin{lemma}[\cite{bp07}]
Let $f \ge 0 \in L^1(\R^2)$ such that $f\log f \in L^1(\R^2)$. Set $\int_{\R^2} fdx=M,$ then
\begin{align}\label{logHLS}
\int_{\R^2} f \log f dx+\frac{2}{M} \iint_{\R^2 \times \R^2} f(x)f(y)\log |x-y| dxdy \ge M\left( \log M-1-\log \pi \right).
\end{align}
\end{lemma}
Now we show the main results for $M_0 \le 8\pi$.
\begin{theorem}[Global existence for $M_0 \le 8\pi$]\label{globalexist}
Assume $\int_{\R^2} |x|^2 u_0(x) dx<\infty$ and $\int_{\R^2} u_0(x) \left|\log u_0(x)\right| dx<\infty.$ When the initial data satisfies $u_0 \in L_+^1 \cap L^\infty(\R^2)$ and
\begin{align}
\int_{\R^2} u_0(x) dx <M_0 \le 8\pi,
\end{align}
then there exists a global weak solution to \er{fisherks} satisfying that for any $0<t<\infty$
\begin{align}
\int_{\R^2} u(\cdot,t) \left|\log u(\cdot,t) \right|dx+\int_{\R^2} u(\cdot,t) |x|^2 dx &\le C\left(e^t,\int_{\R^2} u_0(x) |\log u_0(x)|dx, \int_{\R^2} u_0(x) |x|^2 dx  \right), \\
\|u(\cdot,t)\|_{L^p(\R^2)} &\le C\left(e^t, \|u_0\|_{L^p(\R^2)}\right),\quad 1 \le p \le \infty.
\end{align}
Furthermore, the weak solution satisfies the entropy dissipation
\begin{align}\label{Fu}
 & F[u(\cdot,t)]+\int_0^t \int_{\R^2} \frac{u(x,s)}{m(s)} \left|\nabla \log u(x,s)-\nabla c(x,s) \right|^2 dxds \nonumber \\
 & +\int_0^t \frac{M_0-m(s)}{2m(s)} \int_{\R^2} u(x,s) c(x,s) dxds  \le  F[u_0(x)],
\end{align}
where the free energy
\begin{align}\label{freeenergy}
F(u)=\frac{1}{m(t)} \int_{\R^2} u \log u dx-\frac{1}{2m(t)} \int_{\R^2} u c dx-\log m(t)
\end{align}
and $m(t)=\frac{M_0}{1+\frac{M_0-m_0}{m_0}e^{-M_0 t}}$.
\end{theorem}
\begin{proof}
\noindent We decompose the proof into two parts. In Steps 1-8, a regularized equation is constructed and a priori estimates are established to obtain the global existence of a weak solution to \er{fisherks}. Furthermore, Steps 9-10 give the second moment and the energy inequality of the weak solution.

Following the method of \cite{JLV1}, we take a cut-off function $0
\le \psi_1(x) \le 1 $
\begin{align*}
\psi_1(x)=\left\{
      \begin{array}{ll}
      1 & \mbox{ if } |x|\le 1, \\
      0  & \mbox{ if } |x|\ge 2,
      \end{array}\right.
\end{align*}
where $\psi_1(x) \in C_0^\infty(\R^d)$. Define
\begin{align}\label{cutoff}
\psi_R(x):=\psi_1(x/R),
\end{align}
as $R \to \infty, ~\psi_R \to 1$, then
there exist constants $C_1,C_2$ such that $ |\nabla \psi_R(x)| \le
\frac{C_1}{R},~~|\Delta \psi_R(x)| \le \frac{C_2}{R^2}.$ This cut-off function will be used to derive the
existence of the weak solution. \\

\noindent {\it\textbf{Step 1}} (Approximated problem) \quad In order to show the existence of a weak solution, we firstly consider the regularized problem for $\varepsilon>0$
\begin{align}\label{regularization}
\left\{
  \begin{array}{ll}
    \partial_t u_\varepsilon=\Delta u_\varepsilon-\nabla \cdot (u_\varepsilon c_\varepsilon)+u_\varepsilon \left( M_0-\int_{\R^2} u_\varepsilon dx \right), &  x\in \R^2,t \ge 0,\\
    -\Delta c_\varepsilon=J_\varepsilon \ast u_\varepsilon, & x\in \R^2,t \ge 0,\\
    u_\varepsilon(x,0)=u_{0\varepsilon}(x), & x\in \R^2.
  \end{array}
\right.
\end{align}
Here $J_\varepsilon$ is the regularizing kernel
\begin{align}
J_\varepsilon(x)=\frac{1}{\pi} \frac{\varepsilon^2}{\left( |x|^2+\varepsilon^2 \right)^2}
\end{align}
with $\int_{\R^2} J_\varepsilon dx=1.$ Simple computations show that
\begin{align}
c_\varepsilon(x,t)=-\frac{1}{4\pi} \int_{\R^2} \log \left( |x-y|^2+\varepsilon^2 \right) u_\varepsilon(y,t)dy.
\end{align}
The regularized initial data $u_{0\varepsilon}\in C^\infty(\R^2)$ is a sequence of approximation for $u_0(x)$ and satisfies
\begin{align}\label{u0eps}
\left\{
  \begin{array}{ll}
    \|u_{0\varepsilon}\|_{L^p(\R^2)} \le \|u_0\|_{L^p(\R^2)}~~ \mbox{for~any}~~1 \le p \le \infty, \\[1mm]
  u_{0\varepsilon}(x) \to u_0(x)~~ \mbox{in}~~ L^q(\R^2)~~ \mbox{for}~~1\le q <\infty, \\[1mm]
 \int_{\R^2} |x|^2 u_{0\varepsilon}(x) dx \to \int_{\R^2} |x|^2 u_{0}(x) dx~~\mbox{as}~~\varepsilon \to 0, \\[1mm]
\int_{\R^2} u_{0\varepsilon} |\log u_{0\varepsilon}| dx \to \int_{\R^2}  u_{0} |\log u_{0}|dx~~\mbox{as}~~\varepsilon \to 0. \\
  \end{array}
\right.
\end{align}
Then following the standard parabolic theory, the system on $(u_\varepsilon,c_\varepsilon)$ admits a unique smooth solution with fast decay in space, when the initial data is regularized and truncated. Moreover, there exists $T>0$ such that for $0<t<T$ and any $1 \le r \le \infty$
\begin{align*}
u_\varepsilon \in L^\infty(0,T;L^r(\R^2)).
\end{align*}

\noindent{\it\textbf{Step 2}} (Global $L^1$ norm of $u_\varepsilon$ and the boundedness of the second moment) \quad In this step, we will provide some a priori estimates for $u_\varepsilon$. Firstly, multiplying \er{regularization} with any test function $\psi_R(x)$ gives
\begin{align}\label{psiR}
&\frac{d}{dt} \int_{\R^2} u_\varepsilon(\cdot,t)\psi_R(x)dx \nonumber\\
=&\int_{\R^2} u_\varepsilon \Delta \psi_R(x) dx -\frac{1}{4\pi} \iint_{\R^2 \times \R^2} \frac{[\nabla \psi_R(x)-\nabla \psi_R(y)]\cdot (x-y)}{|x-y|^2+\varepsilon^2} u_\varepsilon(x,t)u_\varepsilon(y,t)dxdy \nonumber\\
&+\int_{\R^2} u_\varepsilon \psi_R(x) dx \left( M_0-\int_{\R^2}u_\varepsilon dx \right).
\end{align}
Letting $\psi_R(x)$ be defined as in \er{cutoff}, then
\begin{align*}
&\left| \int_{\R^2} u_\varepsilon \Delta \psi_R(x) dx \right| \le \frac{C}{R^2} \int_{R^2} u_\varepsilon dx, \\
&\left|\iint_{\R^2 \times \R^2} \frac{[\nabla \psi_R(x)-\nabla \psi_R(y)]\cdot (x-y)}{|x-y|^2+\varepsilon^2} u_\varepsilon(x,t)u_\varepsilon(y,t) dxdy\right| \le \frac{C}{R^2} \left(\int_{R^2} u_\varepsilon dx\right)^2.
\end{align*}
Passing to the limit $R \to \infty$ arrives at
\begin{align}
\frac{d}{dt} \int_{\R^2} u_\varepsilon dx =\int_{\R^2} u_\varepsilon dx \left( M_0-\int_{\R^2}u_\varepsilon dx \right)
\end{align}
and we compute that for any $t>0$
\begin{align}\label{mmt}
m(t):=\int_{\R^2} u_\varepsilon dx=\frac{M_0}{1+ \frac{M_0-m_0}{ m_0} e^{-M_0 t}}.
\end{align}

Now we will prove that the second moment is bounded in time provided the bounded initial second moment $\int_{\R^2} u_\varepsilon(x,0) dx$. Consider a test function $\psi_R(x) \in C_0^\infty(\R^2)$ that grows nicely to $|x|^2$ as $R \to \infty$. We can infer from \er{psiR} that $\int_{\R^2} u_\varepsilon(\cdot,t)\psi_R(x)dx$ is uniformly bounded, thus we may pass to the limit using the Lebesgue monotone convergence theorem to obtain
\begin{align*}
\frac{d}{dt} \int_{R^2} |x|^2 u_\varepsilon dx=&4 \int_{\R^2} u_\varepsilon dx-\frac{1}{2\pi} \iint_{\R^2 \times \R^2} \frac{|x-y|^2}{|x-y|^2+\varepsilon^2} u_\varepsilon(x,t)u_\varepsilon(y,t)
 dxdy \\
& +\int_{\R^2}|x|^2 u_\varepsilon dx \left( M_0-\int_{\R^2}u_\varepsilon dx \right) \\
\le & 4 \int_{\R^2} u_\varepsilon dx+M_0 \int_{R^2} |x|^2 u_\varepsilon dx.
\end{align*}
Then one gets by integrating from $0$ to $t$ in time
\begin{align}\label{second2}
\int_{R^2} |x|^2 u_\varepsilon dx \le C\left(e^t,\int_{\R^2}|x|^2 u_\varepsilon(x,0)dx \right).
\end{align}

\noindent{\it\textbf{Step 3}} (Boundedness of $\int_{\R^2} u_\varepsilon |\log u_\varepsilon|dx$ for $M_0 \le 8\pi$) \quad Recalling \er{mmt} we shall use the notation
\begin{align}
F(u_\varepsilon)=\int_{\R^2} \frac{u_\varepsilon}{m(t)} \log \frac{u_\varepsilon}{m(t)} dx-\frac{1}{2}\int_{\R^2} \frac{u_\varepsilon}{m(t)} c_\varepsilon dx.
\end{align}
Then a straightforward computation shows that
\begin{align}\label{Fueps}
\frac{d}{dt}F(u_\varepsilon)=-\int_{\R^2} \frac{u_\varepsilon}{m(t)} \left| \nabla\left( \log \frac{u_\varepsilon}{m(t)} -c_\varepsilon \right) \right|^2 dx-\frac{M_0-m(t)}{2m(t)} \int_{\R^2} u_\varepsilon c_\varepsilon dx.
\end{align}
Thus for $m_0<M_0,$ differentiating it with respect to $t$ yields
\begin{align}
F(u_\varepsilon) \le F(u_{0\varepsilon}).
\end{align}
It can be checked that
\begin{align}\label{uepslogueps}
\int_{\R^2} u_\varepsilon \log u_\varepsilon dx -\frac{1}{2} \int_{\R^2} u_\varepsilon c_\varepsilon dx \le m(t)F(u_{0\varepsilon})+m(t) \log m(t).
\end{align}
Hence from \er{logHLS} we obtain the inequality
\begin{align}
& \int_{\R^2} u_\varepsilon \log u_\varepsilon dx \nonumber \\
\le & m(t)F(u_{0\varepsilon})+m(t) \log m(t)  -\frac{1}{8\pi} \iint_{\R^2 \times \R^2} u_\varepsilon(x,t) u_\varepsilon(y,t) \log (|x-y|^2+\varepsilon^2) dxdy \nonumber \\
\le & m(t)F(u_{0\varepsilon})+m(t) \log m(t)-\frac{1}{4\pi} \iint_{\R^2 \times \R^2} u_\varepsilon(x,t) u_\varepsilon(y,t) \log |x-y|  dxdy \nonumber \\
\le & m(t)F(u_{0\varepsilon})+m(t) \log m(t)-\frac{m^2(t)}{8 \pi} \left(\log m(t)-1-\log \pi \right)+\frac{m(t)}{8\pi} \int_{\R^2} u_\varepsilon \log u_\varepsilon dx, \label{qingyun}
\end{align}
For $m(t)<M_0 \le  8\pi,$ it follows from \er{mmt} that $\int_{\R^2} u_\varepsilon \log u_\varepsilon dx$ can be estimated respectively
\begin{align}\label{ulogu}
\int_{\R^2} u_\varepsilon \log u_\varepsilon dx \le \left\{
                                                      \begin{array}{ll}
                                                      \frac{8\pi M_0}{m_0(8\pi-M_0)} ~C\left( F(u_{0\varepsilon}),M_0,m_0 \right), & \mbox{for}~M_0<8\pi, \\[2mm]
                                             \frac{8\pi \left( m_0 e^{M_0 t}+M_0-m_0  \right)}{8\pi (M_0-m_0)} ~C\left( F(u_{0\varepsilon}),M_0,m_0 \right), & \mbox{for}~M_0=8\pi.
                                                      \end{array}
                                                    \right.
\end{align}
On the other hand, the following holds
\begin{align}
&1_{ \{\varphi(x)\le 1\} } \varphi(x) \left|\log \varphi(x)\right| \\
=& 1_{ \{\varphi(x)<e^{-2|x|^2} \}} \varphi(x) \log \frac{1}{\varphi(x)} +1_{\{e^{-2|x|^2} \le \varphi(x)\le 1\} } \varphi(x) \log \frac{1}{\varphi(x)} \nonumber\\
\le & 1_{\{\varphi(x)<e^{-2|x|^2} \} } \sqrt{\varphi(x)} +1_{ \{e^{-2|x|^2} \le \varphi(x)\le 1 \} } 2 \varphi(x)|x|^2 \nonumber \\
\le & e^{-|x|^2} +2\varphi(x)|x|^2. \label{xiaoyu1}
\end{align}
Hence combining \er{ulogu} with \er{second2} it gives our result for this step
\begin{align}\label{absulogu}
\int_{\R^2} u_\varepsilon \left|\log u_\varepsilon\right|dx & = \int_{\R^2} u_\varepsilon \log u_\varepsilon dx+2\int_{\R^2} 1_{\{u_\varepsilon \le 1\} } u_\varepsilon \left|\log u_\varepsilon  \right| dx \nonumber \\
& \le \int_{\R^2} u_\varepsilon \log u_\varepsilon dx+2 \int_{\R^2} e^{-|x|^2} dx+4 \int_{\R^2}|x|^2 u_\varepsilon dx \nonumber \\
& \le C_1+C_2(e^t).
\end{align}

\noindent{\it\textbf{Step 4}} (Equi-integrabiliy in $L^p$ norms for $1< p<\infty$) \quad Testing \er{regularization} with $p(u_\varepsilon-K)_+^{p-1}$ we obtain
\begin{align}\label{uK}
&\frac{d}{dt} \int_{\R^2} (u_\varepsilon-K)_+^p dx+\frac{4(p-1)}{p} \int_{\R^2} \left|\nabla (u_\varepsilon-K)_+^{p/2} \right|^2 dx \nonumber\\
=& (p-1) \int_{\R^2} J_\varepsilon \ast u_\varepsilon (u_\varepsilon-K)_+^p dx+Kp\int_{\R^2} J_\varepsilon \ast u_\varepsilon (u_\varepsilon-K)_+^{p-1} dx+p \int_{\R^2} (u_\varepsilon-K)_+^{p-1}u_\varepsilon dx \left(M_0-\int_{\R^2} u_\varepsilon dx \right) \nonumber \\
=&(p-1) \int_{\R^2} J_\varepsilon \ast (u_\varepsilon-K+K) (u_\varepsilon-K)_+^p dx+Kp\int_{\R^2} J_\varepsilon \ast (u_\varepsilon-K+K) (u_\varepsilon-K)_+^{p-1} dx \nonumber \\
& +p \int_{\R^2} (u_\varepsilon-K)_+^{p-1}(u_\varepsilon-K+K) dx \left(M_0-\int_{\R^2} u_\varepsilon dx \right) \nonumber \\
\le& (p-1) \int_{\R^2} (u_\varepsilon-K)_+^{p+1} dx+K(2p-1)\int_{\R^2} (u_\varepsilon-K)_+^p dx+K^2 p \int_{\R^2} (u_\varepsilon-K)_+^{p-1} dx\nonumber \\
& +M_0 p \int_{\R^2}(u_\varepsilon-K)_+^p dx+KpM_0\int_{\R^2} (u_\varepsilon-K)_+^{p-1} dx.
\end{align}
We use GNS inequality to get
\begin{align}\label{11}
\int_{\R^2}(u_\varepsilon-K)_+^{p+1} dx\le C(p) \int_{\R^2} \left|\nabla (u_\varepsilon-K)_+^{p/2} \right|^2 dx \int_{\R^2} (u_\varepsilon-K)_+ dx.
\end{align}
By \er{absulogu} one has
\begin{align}\label{22}
\int_{u_\varepsilon>K} (u_\varepsilon-K)_+dx \le \int_{u_\varepsilon>K} u_\varepsilon dx \le \frac{\int_{u_\varepsilon>K} u_\varepsilon |\log u_\varepsilon| dx}{\log K} \le \frac{\int_{\R^2} u_\varepsilon |\log u_\varepsilon| dx}{\log K} \le \frac{C(t)}{\log K}.
\end{align}
Moreover, $\int_{\R^2} (u_\varepsilon-K)_+^{p-1} dx$ can be estimated as
\begin{align}\label{33}
\int_{\R^2} (u_\varepsilon-K)_+^{p-1}dx &=\int_{K\le u_\varepsilon \le K+1 } (u_\varepsilon-K)_+^{p-1}dx+\int_{u_\varepsilon>K+1} (u_\varepsilon-K)_+^{p-1}dx \nonumber\\
& \le \int_{K\le u_\varepsilon \le K+1} 1 dx+ \int_{u_\varepsilon>K+1} (u_\varepsilon-K)_+^{p}dx \nonumber \\
& \le \frac{\int_{\R^2}u_\varepsilon dx}{K}+\int_{\R^2} (u_\varepsilon-K)_+^{p}dx.
\end{align}
Collecting \er{11}, \er{22} and \er{33} we can go further from \er{uK} that
\begin{align}
& \frac{d}{dt} \int_{\R^2} (u_\varepsilon-K)_+^p dx+\frac{4(p-1)}{p} \int_{\R^2} \left|\nabla (u_\varepsilon-K)_+^{p/2} \right|^2 dx \nonumber\\
\le & C(p) \frac{C(t)}{\log K}\int_{\R^2} \left|\nabla (u_\varepsilon-K)_+^{p/2} \right|^2 dx +C(K,p) \frac{\int_{\R^2}u_\varepsilon dx}{K}+C(K,p) \int_{\R^2} (u_\varepsilon-K)_+^p dx.
\end{align}
Choosing $K$ large enough such that $\frac{C(p)C(t)}{\log K}<\frac{4(p-1)}{p}$, we obtain
\begin{align}\label{nablaup}
\frac{d}{dt} \int_{\R^2} (u_\varepsilon-K)_+^p dx +C(p,K) \int_{\R^2} \left|\nabla (u_\varepsilon-K)_+^{p/2} \right|^2 dx  \le C_1+C_2 \int_{\R^2} (u_\varepsilon-K)_+^p dx
\end{align}
and subsequently for any $0<t<T$
\begin{align}\label{uKp}
\int_{\R^2} (u_\varepsilon-K)_+^p dx \le C\left(e^t,\int_{\R^2} u_{0\varepsilon}^p dx \right).
\end{align}

\noindent{\it\textbf{Step 5}} (Boundedness of $L^p$ norm for $1< p \le \infty$) \quad Now we can go further to prove that $\int_{\R^2} u_\varepsilon^p dx$ is bounded as follows.
\begin{align}
\int_{\R^2} u_\varepsilon^p dx =& \int_{u_\varepsilon \le K} u_\varepsilon^p dx+\int_{u_\varepsilon>K} u_\varepsilon^p dx \nonumber \\
\le & K^{p-1} \int_{\R^2}u_\varepsilon dx+\int_{u_\varepsilon>K} (u_\varepsilon-K+K)^{p-1}u_\varepsilon dx \nonumber \\
\le & K^{p-1}M_0+\max(2^{p-2},1) \left( \int_{u_\varepsilon>K} (u_\varepsilon-K)^{p-1}u_\varepsilon dx+K^{p-1}M_0 \right) \nonumber \\
\le & K^{p-1}M_0+\max(2^{p-2},1) \left( K \int_{\R^2} (u_\varepsilon-K)^{p-1}dx+\int_{\R^2} (u_\varepsilon-K)^{p}dx  +K^{p-1}M_0 \right) \nonumber \\
\le & K^{p-1}M_0+\max(2^{p-2},1) \left(\int_{\R^2}u_\varepsilon dx+ K \int_{\R^2} (u_\varepsilon-K)^{p}dx+\int_{\R^2} (u_\varepsilon-K)^{p}dx  +K^{p-1}M_0 \right), \nonumber
\end{align}
where the last line is given by \er{33}. Therefore, \er{uKp} guarantees that
for any $0<t<T$
\begin{align}\label{Lp}
\int_{\R^2} u_\varepsilon^p dx \le C(M_0,e^t,K,p)
\end{align}
for any $1<p<\infty.$ Furthermore, integrating \er{nablaup} from $0$ to $T$ we have that for any $T>0$
\begin{align}\label{nablaLp}
\nabla u_\varepsilon^{\frac{p}{2}} \in L^2\left(0,T;L^2(\R^2)\right)~~\mbox{for~any}~~1<p<\infty.
\end{align}
Finally, making use of the iterative method in the spirit of \cite{BL14} we have the uniformly boundedness
\begin{align}\label{Linfinity}
u_\varepsilon \in L^\infty\left(0,T;L^\infty(\R^2)\right).
\end{align}

\noindent{\it\textbf{Step 6}} (Time regularity) \quad It directly follows from \er{Lp}, \er{nablaLp} and \er{Linfinity} that
\begin{align}
& \|\nabla u_\varepsilon^r \|_{L^2(0,T;L^2(\R^2))} \le C, \quad \mbox{for~any}~~r > 1/2, \\
& \|u_\varepsilon \nabla c_\varepsilon\|_{L^\infty(0,T;L^\infty(\R^2))} \le C,  \label{uc}\\
& \| \partial_t u_\varepsilon \|_{L^2(0,T;H^{-1}(\R^2))} \le C. \label{ut}
\end{align}
Then by Lemma 4.23 in \cite{CY19} one has that for any bounded domain $\Omega$,
there exists a subsequence $u_\varepsilon$ without relabeling such that
\begin{align}
u_\varepsilon \to u~~\mbox{in}~~L^2\left(0,T;L^{\bar{p}}(\Omega)\right),\quad \mbox{for~any}~1 \le \bar{p}<\infty.
\end{align}
By a standard diagonal argument, the following uniform strong convergence holds true that for any $R>0$
\begin{align}\label{123456}
u_\varepsilon \to u ~~\mbox{in}~~ L^2\left(0,T;L^{\bar{p}}(B_R)\right).
\end{align}

Let us turn back to sketch the proof of \er{uc} and \er{ut}. Actually, the term
\begin{align}
\nabla c_\varepsilon=-\frac{1}{2\pi} \int_{\R^2} \frac{x-y}{|x-y|^2+\varepsilon^2} u_\varepsilon(y,t) dy
\end{align}
can be estimated from the Young inequality that
\begin{align}
\|\nabla c_\varepsilon\|_{L^\infty(\R^2)} & \le \frac{1}{2\pi} \left\| \int_{0<|x-y|\le 1} \frac{u_\varepsilon(y,t)}{|x-y|} dy+ \int_{|x-y|> 1} \frac{u_\varepsilon(y,t)}{|x-y|} dy  \right\|_{L^\infty(\R^2)} \nonumber\\
& \le \frac{1}{2\pi}\left( \|u_\varepsilon\|_{L^\infty(\R^2)} \left\| \frac{1}{|x|} \right\|_{L^1(0<|x|\le 1)}+\|u_\varepsilon\|_{L^1(\R^2)} \right) \nonumber \\
& \le C \left( \|u_\varepsilon\|_{L^\infty(\R^2)}+\|u_\varepsilon\|_{L^1(\R^2)} \right).
\end{align}
The bound of $u_\varepsilon \nabla c_\varepsilon \in L^\infty(0,T;L^\infty(\R^2))$ follows by using $u_\varepsilon \in L^\infty(0,T;L^1(\R^2))$.

Moreover, \er{ut} can be deduced by using the previous estimates that for any test function $\psi \in L^2\left(0,T;H^1(\R^2)\right)$
\begin{align}
\left| \int_0^T \int_{\R^2} \partial_t u_\varepsilon ~\psi dxdt \right| & \le \int_0^T \int_{\R^2} \left|\nabla u_\varepsilon \cdot \nabla \psi\right|dxdt+
\int_0^T \int_{\R^2} \left|u_\varepsilon \nabla c_\varepsilon \cdot \nabla \psi  \right|  dxdt \nonumber \\
& \le \left( \|\nabla u_\varepsilon\|_{L^2(0,T;\R^2)}+\|\sqrt{u_\varepsilon} \nabla c_\varepsilon\|_{L^\infty(0,T;L^\infty(\R^2))} \sqrt{T \|u_\varepsilon\|_{L^1(\R^2)}} ~\right) \|\nabla \psi\|_{L^2(0,T;L^2(\R^2))} \nonumber\\
& \le C \|\nabla \psi\|_{L^2(0,T;L^2(\R^2))}.
\end{align}

\noindent{\it\textbf{Step 7}} (Strong convergence of $u_\varepsilon$) \quad Furthermore, we will take advantage of the second moment estimate \er{second2} to establish the strong convergence of $u_\varepsilon$ in $L^2(0,T;L^q(\R^2))$ and extend \er{123456} to the whole space. We compute that for any $1 \le q<\infty$
\begin{align}\label{farfield}
\int_0^T \|u_\varepsilon\|_{L^q(|x|>R)}^2 dt & \le \int_0^T \|u_\varepsilon\|_{L^\infty(|x|>R)}^{2(q-1)/q} \|u_\varepsilon\|_{L^1(|x|>R)}^{2/q} dt \nonumber \\
& \le \int_0^T \|u_\varepsilon\|_{L^\infty(|x|>R)}^{2(q-1)/q} \frac{ \left(\int_{\R^2}|x|^2 u_\varepsilon dx\right)^{2/q} }{R^{4/q}} dt \to 0~~\mbox{as}~~R \to \infty,
\end{align}
and the weak semi-continuity of $L^2\left(0,T;L^q(\R^2)\right)$ implies
\begin{align}\label{semi}
\int_0^T \|u \|_{L^q(|x|>R)}^2 dt \le \displaystyle \liminf_{\varepsilon \to 0}  \int_0^T \|u_\varepsilon\|_{L^q(|x|>R)}^2 dt \to 0 ~~\mbox{as}~~R \to \infty.
\end{align}
Therefore, the following inequality is derived that for any $1 \le q<\infty,$ as $R \to \infty, \varepsilon \to 0,$
\begin{align}
&\int_0^T \|u_\varepsilon-u\|_{L^q(\R^2)}^2 dt=\int_0^T \left( \|u_\varepsilon-u\|_{L^q(|x|>R)}+\|u_\varepsilon-u\|_{L^q(|x|\le R)}  \right)^2 dt \nonumber \\
\le & ~C(q) \left( \int_0^T \|u_\varepsilon\|_{L^q(|x|>R)}^2+\int_0^T \|u\|_{L^q(|x|>R)}^2 dt+\int_0^T \|u_\varepsilon-u\|_{L^q(|x|\le R)}^2 dt    \right) \to 0.
\end{align}
In the last inequality, the first term goes to zero due to \er{farfield}, the second term is given by \er{semi} and \er{123456} provides the third term. Hence one has
\begin{align}\label{L1strong}
u_\varepsilon \to u ~~\mbox{in}~~ L^2\left(0,T;L^{r}(\R^2)\right),\quad \mbox{for~any}~1\le r<\infty.
\end{align}

\noindent{\it\textbf{Step 8}} (Existence of the weak solution) \quad
Now multiplying \er{regularization} by $\psi \in C_0^\infty(\R^2)$ and integrating it with respect to $x$ and $t$, we get the weak formulation for $u_\varepsilon$
\begin{align}\label{weaksolution}
& \int_{\R^2} \psi u_\varepsilon(\cdot,t) dx-\int_{\R^2} \psi u_{0\varepsilon}dx \nonumber \\
= & \int_0^T \int_{\R^2} \Delta \psi u_\varepsilon dxdt
-\frac{1}{4\pi} \iint_{\R^2 \times \R^2} \frac{[\nabla \psi(x)-\nabla \psi(y)]\cdot(x-y)}{|x-y|^2+\varepsilon^2} u_\varepsilon(x,t) u_\varepsilon(y,t) dxdydt \nonumber \\
& +\int_0^T \int_{\R^2} u_\varepsilon(\cdot,t) \psi(x) dx \left(M_0-\int_{\R^2} u_\varepsilon(\cdot,t) dx \right) dt.
\end{align}
It's to be noticed that \er{123456} directly yields
\begin{align}\label{ubdd}
\int_0^T \int_{\R^2} \Delta \psi u_\varepsilon dxdt \to \int_0^T \int_{\R^2} \Delta \psi u dxdt,\quad \varepsilon \to 0.
\end{align}
As to the second term of the right side of \er{weaksolution}, note that
\begin{align}
&\left| \int_0^T \iint_{\R^2\times \R^2} [\nabla \psi(x)-\nabla \psi(y)]\cdot(x-y) \left( \frac{1}{|x-y|^2}-\frac{1}{|x-y|^2+\varepsilon^2} \right) u_\varepsilon(x,t)u_\varepsilon(y,t) dxdydt \right| \nonumber\\
\le &C \varepsilon \int_0^T \iint_{\R^2 \times \R^2} \frac{u_\varepsilon(x,t)u_\varepsilon(y,t)}{|x-y|}dxdydt \le C \varepsilon \int_0^T \|u_\varepsilon\|_{L^{4/3}(\R^2)}^2 dt \le C \varepsilon. \label{eps2}
\end{align}
In addition, by Cauchy inequality we have
\begin{align}
&\left| \int_0^T \iint_{\R^2 \times \R^2} [\nabla \psi(x)-\nabla \psi(y)]\cdot(x-y) \left( \frac{u_\varepsilon(x,t)u_\varepsilon(y,t)}{|x-y|^2}-\frac{u(x,t)u(y,t)}{|x-y|^2}  \right) dxdy dt \right| \nonumber \\
\le& C \left( \int_0^T \iint_{\Omega \times \Omega} |u_\varepsilon(x)-u(x)|u_\varepsilon(y) dxdy+\iint_{\Omega \times \Omega} |u_\varepsilon(y)-u(y)|u(x) dxdydt  \right) \nonumber \\
\le &C \int_0^T \|u_\varepsilon-u\|_{L^2(\Omega)}^2 dt \int_0^T \|u_\varepsilon\|_{L^2(\Omega)}^2 dt, \label{appro}
\end{align}
where the last line is given by the semi-continuity of $\|u\|_{L^2(\R^2)} \le \displaystyle \liminf_{\varepsilon \to 0} \|u_\varepsilon\|_{L^2(\R^2)}$. Thus taking the limit $\varepsilon \to 0$ and combining \er{eps2} and \er{appro} we conclude
\begin{align}\label{epsappro}
\int_0^T \iint_{\R^2 \times \R^2} [\nabla \psi(x)-\nabla \psi(y)]\cdot(x-y) \left( \frac{u_\varepsilon(x,t)u_\varepsilon(y,t)}{|x-y|^2+\varepsilon^2}- \frac{u(x,t)u(y,t)}{|x-y|^2} \right) dxdydt \to 0.
\end{align}
Thanks to \er{ubdd}, \er{epsappro} and \er{L1strong}, passing to the limit $\varepsilon \to 0$ in \er{weaksolution} one has that for any $0<t<T$
\begin{align}\label{weak}
&\int_{\R^2} \psi u(\cdot,t) dx-\int_{\R^2} \psi u_{0}(x)dx \nonumber \\
= &\int_0^T \int_{\R^2} \Delta \psi u dxdt -\frac{1}{4\pi} \iint_{\R^2 \times \R^2} \frac{[\nabla \psi(x)-\nabla \psi(y)]\cdot(x-y)}{|x-y|^2} u(x,t) u(y,t) dxdydt \nonumber \\
& +\int_0^T \int_{\R^2} u(\cdot,t) \psi(x) dx \left(M_0-\int_{\R^2} u(\cdot,t) dx \right) dt.
\end{align}
This gives the existence of a global weak solution.

\noindent{\it\textbf{Step 9}} (The second moment of the weak solution) \quad
Consider a test function $\psi_R(x) \in C_0^\infty(\R^2)$ and $\psi_R(x)=|x|^2$ for $|x|<R, \psi_R(x)=0$ for $|x| \ge 2R,$ letting $\psi=\psi_R$ in \er{weak} arrives at
\begin{align}\label{secondweak1}
&\int_{\R^2} \psi_R u(\cdot,t) dx-\int_{\R^2} \psi_R u_{0}(x)dx \nonumber \\
= &\int_0^T \int_{\R^2} \Delta \psi_R u dxdt -\frac{1}{4\pi} \iint_{\R^2 \times \R^2} \frac{[\nabla \psi_R(x)-\nabla \psi_R(y)]\cdot(x-y)}{|x-y|^2} u(x,t) u(y,t) dxdydt \nonumber \\
& +\int_0^T \int_{\R^2} u(\cdot,t) \psi_R(x) dx \left(M_0-\int_{\R^2} u(\cdot,t) dx \right) dt.
\end{align}
Moreover, we follow the same lines as \er{psiR} to \er{mmt} to claim that
\begin{align}
\int_{\R^2} u(x,t) dx=m(t).
\end{align}
As before, since $\Delta \psi_R(x)$ and $\frac{[\nabla \psi_R(x)-\nabla \psi_R(y)]\cdot(x-y)}{|x-y|^2}$ are bounded, thus the first two terms in the right-hand side of \er{secondweak1} are bounded. As a consequence, as $R \to \infty$, we may pass to the limit using the Lebesgue monotone convergence theorem with $u \in L^1(\R^2)$ and obtain that for any $t>0$
\begin{align}\label{513}
\int_{\R^2} |x|^2 u(x,t) dx = &\int_{\R^2} |x|^2 u_0(x) dx+4\int_0^t m(s)ds-\frac{1}{2\pi} \int_0^t m^2(s) ds \nonumber \\
 & + \int_0^t \int_{\R^2} |x|^2 u(x,s) dx ( M_0-m(s)) ds.
\end{align}
Then by Gronwall's inequality we have
\begin{align}\label{secondweak}
\int_{\R^2} |x|^2 u(x) dx \le C.
\end{align}

\noindent{\it\textbf{Step 10}} (The energy inequality of the weak solution) \quad Integrating \er{Fueps} in time from $0$ to $t$ follows
\begin{align}\label{Fueps1}
&\frac{1}{m(t)} \int_{\R^2} u_\varepsilon \log u_\varepsilon dx-\frac{1}{2m(t)} \int_{\R^2} u_\varepsilon c_\varepsilon dx-\log m(t)+\int_0^t \int_{\R^2} \frac{u_\varepsilon}{m(s)} \left|\nabla \log u_\varepsilon-\nabla c_\varepsilon \right|^2 dxds \nonumber \\
+&\int_0^t \frac{M_0-m(s)}{2m(s)} \int_{\R^2} u_\varepsilon c_\varepsilon dxds = \frac{1}{m_0} \int_{\R^2} u_{0\varepsilon} \log u_{0\varepsilon} dx-\frac{1}{2m_0} \int_{\R^2} u_{0\varepsilon} c_{0\varepsilon} dx-\log m_0.
\end{align}
The aim of the final step is taking limit $\varepsilon \to 0$ in \er{Fueps1} to deduce the energy inequality \er{Fu}.

For the sake of passing limit in \er{Fueps1}, we split it into three parts. Firstly it is derived from Lemma \ref{ulogusemi} in the "Appendix" that
\begin{align}\label{logsemi}
\int_{\R^2} u(x,t) \log u(x,t) dx \le \displaystyle \liminf_{\varepsilon \to 0} \int_{\R^2} u_\varepsilon(x,t) \log u_\varepsilon(x,t) dx.
\end{align}
Secondly, the lower semi-continuity of the energy dissipation is followed from Lemma 4.9 of \cite{CY19} that for any $t>0$
\begin{align}\label{dissilower}
\int_0^t \frac{1}{m(s)} \int_{\R^2} \left|2 \nabla \sqrt{u}-\sqrt{u}\nabla c\right|^2 dxds \le \displaystyle \liminf_{\varepsilon \to 0} \int_0^t \frac{1}{m(s)} \int_{\R^2} \left|2 \nabla \sqrt{u_\varepsilon}-\sqrt{u_\varepsilon}\nabla c_\varepsilon\right|^2 dxds.
\end{align}
It remains to verify the strong convergence of $\int_{\R^2}u_\varepsilon c_\varepsilon dx$ that
\begin{align}\label{stronguc}
\int_{\R^2} u_\varepsilon(x,t) c_\varepsilon(x,t) dx \to \int_{\R^2} u(x,t) c(x,t) dx ~~\mbox{a.e.~in}~~(0,T).
\end{align}
We write
\begin{align}
& -4 \pi \int_{\R^2} u_\varepsilon(x,t) c_\varepsilon(x,t)-u(x,t)c(x,t) dx \nonumber\\
=& \iint_{\R^2 \times \R^2} u_\varepsilon(x,t) u_\varepsilon(y,t) \log \frac{|x-y|^2+\varepsilon^2}{|x-y|^2} dxdy +2 \iint_{\R^2 \times \R^2} u_\varepsilon(x,t) ( u_\varepsilon(y,t)-u(y,t)) \log|x-y| dxdy \nonumber \\
& +2\iint_{\R^2 \times \R^2} (u_\varepsilon(x,t)-u(x,t))u(y,t) \log |x-y| dxdy =:I_1+I_2+I_3. \label{I1I2I3}
\end{align}
For $I_1,$ we compute that for any $0<t<T$
\begin{align}\label{I426}
I_1=& \iint_{\{|x-y| \ge 1\}} u_\varepsilon(x,t) u_\varepsilon(y,t) \log \frac{|x-y|^2+\varepsilon^2}{|x-y|^2} dxdy+\iint_{\{|x-y| <1\}} u_\varepsilon(x,t) u_\varepsilon(y,t) \log \frac{|x-y|^2+\varepsilon^2}{|x-y|^2} dxdy \nonumber \\
\le &\iint_{\{|x-y| \ge 1\}} u_\varepsilon(x,t) u_\varepsilon(y,t) \log (1+\varepsilon^2) dxdy  \nonumber \\
& +C \|u_\varepsilon\|_{L^\infty(\R^2)} \|u_\varepsilon\|_{L^1(\R^2)} \int_0^1 r \log \frac{r^2+\varepsilon^2}{r^2} dr \to 0, ~~\mbox{as}~~\varepsilon \to 0.
\end{align}
Moreover, we deal with $I_2$ or $I_3$ as follows.
\begin{align*}
\frac{1}{2}|I_2| \le & \iint_{|x-y| \ge 1} u_\varepsilon(x,\cdot) |u_\varepsilon(y,\cdot)-u(y,\cdot)| \log|x-y| dxdy \\
& + \iint_{|x-y| < 1} u_\varepsilon(x,\cdot) |u_\varepsilon(y,\cdot)-u(y,\cdot)| \left|\log|x-y| \right| dxdy=:II_1+II_2.
\end{align*}
Here it can be estimated that as $\varepsilon \to 0$
\begin{align}
II_1 = & \iint_{|x-y| \ge 1} u_\varepsilon(x,\cdot) \left|u_\varepsilon(y,\cdot)-u(y,\cdot)\right|^{1/2} \left|u_\varepsilon(y,\cdot)-u(y,\cdot)\right|^{1/2} \log|x-y| dxdy \nonumber\\
\le & \iint_{|x-y| \ge 1} u_\varepsilon(x,\cdot)^{1/2} u_\varepsilon(x,\cdot)^{1/2} \left|u_\varepsilon(y,\cdot)-u(y,\cdot)\right|^{1/2} \left|u_\varepsilon(y,\cdot)-u(y,\cdot)\right|^{1/2} |x-y| dxdy  \nonumber \\
\le & C\left(\|u_\varepsilon\|_{L^1(\R^2)}\right) \|u_\varepsilon-u\|_{L^1(\R^2)}^{1/2} \int_{\R^2} u_\varepsilon(x)|x|^2 dx \to 0 \label{I1greater}
\end{align}
and
\begin{align}\label{I2less}
II_2 \le & \iint_{|x-y| < 1} u_\varepsilon(x,\cdot) \left|u_\varepsilon(y,\cdot)-u(y,\cdot)\right| \log \frac{1}{|x-y|} dxdy \nonumber \\
\le & C \|u_\varepsilon-u\|_{L^1(\R^2)} \|u_\varepsilon\|_{L^\infty(\R^2)} \int_0^1 r \log \frac{1}{r} dr \to 0.
\end{align}
Getting \er{I1greater} and \er{I2less} together and handling $I_3$ similarly yield that as $\varepsilon \to 0$,
\begin{align}
I_2 \to 0, \quad I_3 \to 0.
\end{align}
Thus \er{stronguc} can be obtained from \er{I1I2I3}.

Hence combining \er{logsemi}, \er{dissilower} and \er{stronguc}, letting $\varepsilon \to 0$ in \er{Fueps1} leads to
\begin{align*}
&\int_{\R^2} \frac{u(x,t)}{m(t)} \log \frac{u(x,t)}{m(t)} dx-\frac{1}{2m(t)} \int_{\R^2} u(x,t) c(x,t) dx+\int_0^t \int_{\R^2} \frac{u(x,s)}{m(s)} \left|\nabla \log u(x,s)-\nabla c(x,s) \right|^2 dxds \nonumber \\
+&\int_0^t \frac{M_0-m(s)}{2m(s)} \int_{\R^2} u(x,s) c(x,s) dxds \le \frac{1}{m_0} \int_{\R^2} u_{0}(x) \log u_{0}(x) dx-\frac{1}{2m_0} \int_{\R^2} u_{0}(x) c_{0} dx-\log m_0,
\end{align*}
which is the desired inequality \er{Fu} and thus completes the proof. \quad $\Box$
\end{proof}

\par
\begin{remark}
For $M_0<m_0<8\pi,$ then $M_0-m(t)<0$ for all $t>0.$ Therefore, $u(x,t)$ is a sub-solution of the keller-segel system $v_t=\Delta v-\nabla \cdot (v \nabla w),~-\Delta w=v$ which admits a global solution when $\int_{\R^2}v(x,0)dx <8\pi$. By the comparison principle, all solutions of \er{fisherks} exist globally.
\end{remark}

\subsection{Finite time blow-up for $M_0>8\pi$}\label{sec2}

In this subsection, we will show the blow-up statement with finite initial second moment. We start with the case $M_0>8\pi$ in which we use the standard argument relying on the evolution of the second moment of solutions as done in \cite{bdp06,bp07}. By passing to the limit from Steps 1-8 and adapting the argument \er{513} in Step 9 of Theorem \ref{globalexist} without further computation we obtain
\begin{lemma}[The Second Moment]
Assume  $\int_{R^2}|x|^2 u_{0}(x) dx < \infty$ and $u_0(x) \in L_+^1 \cap L^\infty(\R^2)$. Let $u(x,t)$ be the weak solution to system \er{fisherks}, then the the second moment satisfies
\begin{align}\label{m2t}
    m_2(t):=\int_{R^2}|x|^2 u(x,t) dx=m(t)\left[4t-\frac{1}{2\pi}\ln\left(\frac{m_0}{M_0}e^{M_0 t}+\frac{M_0-m_0}{M_0}\right)+\frac{\int_{R^2}|x|^2 u_{0}(x) dx}{m_0}\right]
\end{align}
where $m(t)=\left(\frac{1}{M_0}-\left( \frac{m_0-M_0}{M_0 m_0} \right)e^{-M_0 t}\right)^{-1}.$
\end{lemma}

 We notice that the second moment evolution is more complicated than in the classical system corresponding to mass conservation where the time derivative of the second moment is a constant. An easy consequence of the previous lemma is the following blow-up result.

\begin{theorem}[Finite time blow-up for $M_0>8\pi$] \label{blowup}
Assume $\int_{\R^2} |x|^2 u_0(x) dx<\infty$.  If $M_0>8\pi$, then the solution of \er{fisherks} blows up in finite time and there exists a $0<T<\infty$ such that
\begin{align}
\displaystyle \limsup_{t \to T} \|u(\cdot,t)\|_{p}=\infty,~~\mbox{for any}~~p>1.
\end{align}
\end{theorem}
\begin{proof}
Our proof hinges on the second moment evolution. Firstly for $M_0>8\pi,$ the second moment $\int_{\R^2} |x|^2 u(x,t) dx$  can be estimated that
\begin{align}\label{5133}
\int_{\R^2} |x|^2 u(x,t) dx \le \left\{
                                                      \begin{array}{ll}
                                                      \left(4-\frac{M_0}{2\pi}\right)M_{0}t+\frac{M_0}{m_0}\int_{\R^2} |x|^2 u_0(x) dx-\frac{M_0}{2\pi}\ln\frac{m_0}{M_0}& \mbox{for}~m_0\le M_0, \\[2mm]
                              \left(4-\frac{M_0}{2\pi}\right)m_{0}t+\int_{\R^2} |x|^2 u_0(x) dx               , & \mbox{for}~m_0 > M_0.
                                                      \end{array}
                                                    \right.
\end{align}
Consider $M_0>8\pi$, it enables us to get that the second moment of the weak solution will become negative after some time and contradicts the non-negativity of $u$. Therefore, there is a $T^* > 0$ such that $\displaystyle \lim_{t\to T^*} m_2(t) = 0$, and using H\"older's inequality one has
\begin{align}
   \int_{\R^2} u(x,t) dx =\int_{|x|\le R} u(x,t) dx+\int_{|x|> R} u(x,t) dx \le C R^{2(p-1)/p}||u||_{L^p} + \frac{1}{R^2} m_2(t),~for~all~p>1.
\end{align}
Choosing $R=\left(\frac{Cm_2(t)}{||u||_{L^p}}\right)^{1/(a+2)}$ with $a=2(p-1)/p$, we obtain
\begin{align}
   ||u||_{L^1}\le C ||u||_{L^p}^{\frac{2}{a+2}}m_2(t)^{\frac{a}{a+2}},
\end{align}
which induces that
\begin{align}
     \limsup_{t\to T^*}||u||_{L^p}\ge \lim_{t\to T^*}\frac{||u(x,t)||_{L^1}^{\frac{a+2}{2}}}{Cm_2(t)^{\frac{a}{2}}} =\infty.
\end{align}
Thus  the proof is completed. \quad $\Box$
\end{proof}

In addition, we know that the solution exists globally for $m_0<M_0<8\pi$. On the contrary, we could obtain the finite time blow-up result for $M_0<8\pi<m_0$.
\begin{theorem}[Finite time blow-up for $M_0<8\pi<m_0$] \label{blowup2}
For $M_0 < 8\pi < m_0$, assume $\int_{\R^2} |x|^2 u_0(x) dx<C(M_0,m_0)$, where $C(M_0,m_0)$ is given by
\begin{align*}
    C(M_0,m_0):=-\frac{(8\pi-M_0)m_0}{2M_0\pi}\ln{\frac{m_0-M_0}{8\pi-M_0}}+\frac{4m_0}{M_0}\ln{\frac{m_0}{8\pi}},
\end{align*}
then the solution of \er{fisherks}  blows up in finite time.
\end{theorem}
\begin{proof}
Denote $h(t)=\frac{1}{m(t)} \int_{R^2}|x|^{2}u(x,t) dx$, from \er{m2t} we find,
\begin{align*}
   \frac{d h(t)}{dt}=4-\frac{m(t)}{2\pi}
     ,~~\frac{d^2 h(t)}{dt^2} =-\frac{m(t)\left(M_{0}-m(t)\right)}{2\pi}.
\end{align*}
For $M_0 < 8\pi < m_0$, we are able to find that $h(t)$ firstly decreases from the initial data and reaches the global minimum $\frac{\int_{\R^2}|x|^2 u_0(x) dx-C(M_0,m_0) }{8\pi}$. Thus we conclude that $m_{2}(t)$ should become negative in finite time when the initial second moment is small, namley $\int_{\R^2}|x|^2 u_0(x) dx<C(M_0,m_0) $. This suggests that the solution blows up in finite time as proved in Theorem \ref{blowup}. Therefore the blow-up statement for the case $M_0 < 8\pi < m_0$ is completed without further comment. \quad $\Box$
\end{proof}

\section{Infinite time blow-up for $m_0<M_0= 8\pi$}\label{sec3}
\def\theequation{3.\arabic{equation}}\makeatother
\setcounter{equation}{0}
\def\thetheorem{3.\arabic{theorem}}\makeatother
\setcounter{theorem}{0}

As mentioned earlier, the solution exists globally for $0<m_0<M_0=8\pi$. In this section, we will firstly show that nontrivial steady state solutions can exist only in the case $M_0=8\pi$. The infinite second moment of the steady states could serve to give a hint on the further proof of infinite time blow-up.

\subsection{Steady states}\label{31}

This subsection is primarily devoted to the analysis on the steady solutions of \er{fisherks}. Keeping \er{mt} in mind we say that $m(t) \to M_0$ as $t \to \infty$, and the stationary equation is followed in the sense of distribution
\begin{align}\label{2022}
\left\{
  \begin{array}{ll}
    \Delta U_s(x)-\nabla\cdot (U_s(x)\nabla C_s(x))=0,~~x \in \R^2, \\
    -\Delta C_s(x)=U_s(x),~~x \in \R^2
  \end{array}
\right.
\end{align}
with
\begin{align}\label{M0}
\int_{\R^2} U_s(x) dx=M_0.
\end{align}

Now three equivalent statements for the stationary solutions are shown and one has the constant chemical potential inside the support of the steady solutions.

\begin{proposition}[Three equivalent statements for the steady states]\label{steadyproperties}
Let $\Omega \in \R^2$ be a connected open set.
Assuming that $U_s \log U_s \in L^1(\R^2)$ and $U_s \in L_+^1(\R^2)$ with
$\int_{\R^2} U_s dx=M_0$, $U_s \in C(\bar{\Omega})$ and $U_s>0
~~\mbox{in}~~\Omega,~~U_s=0 ~~\mbox{in}~~\R^2 \setminus \Omega$. Moreover, if $\Omega$ is unbounded, assume that $U_s$ decays at infinity.

Assume also $C_s \in C^2(\R^2)$ is the Newtonian potential
satisfying
\begin{align}
& \Delta U_s-\nabla\cdot (U_s \nabla C_s )=0, ~~\mbox{in}~~\R^2, \label{mus} \\
& \mu_s=\log U_s -C_s, \mbox{ in } \R^2
\label{mupotential}
\end{align}
in the sense of distribution. Then the following three statements are equivalent: \\
(i)  No dissipation: $\int_{\Omega} U_s \big|\nabla \mu_s \big|^2 dx=0$. \\
(ii)  $U_s$ is the minimizer of the total interaction energy $F(u)=\frac{1}{m(t)} \int_{\R^2} u \log u dx-\frac{1}{2m(t)} \int_{\R^2} u c dx-\log m(t)$.\\ 
(iii) The chemical potential satisfies
\begin{align}
\mu_s(x) = \bar{C}, & \qquad \forall x \in \mbox{Supp} (U_{s})
\end{align}
where $\mbox{Supp} (U_{s})=\R^2$ and
\begin{align}\label{Cbar}
\bar{C}=\frac{1}{M_0} \left( \int_{\R^2} U_s \log U_s -U_s c_s dx\right).
\end{align}
\end{proposition}
\begin{proof}
\noindent Firstly we prove (i) $\Leftrightarrow$ (iii): (i) is directly from (iii). (i)$ \Rightarrow $ (iii): Suppose $\int_{\Omega} U_s |\nabla \mu_s|^2 dx=0$. It follows from $U_s>0$ at any point $x_0\in \Omega$ that $\nabla \mu_s=0$ in a neighborhood of $x_0$ and thus $\mu_s$ is a constant in this neighborhood. By the connectedness of $\Omega$ we deduce that $\mu_s \equiv \bar{C}$ in $\Omega$.

In order to prove (ii) $\Leftrightarrow$ (iii), we define the minimizer of $F(u)$: for any $\varphi \in C_0^\infty (\Omega),$ let $\overline{\Omega}_0 =\mbox{ supp } \varphi
$ with $\int_\Omega \varphi(x)dx=0,~\Omega_0 \subset\subset \Omega$.
There exists
$$
 \varepsilon_0: = \frac{ \displaystyle \min_{y \in \overline{\Omega}_0} U_s(y)} { \displaystyle \max_{y \in \overline{\Omega}_0} \left| \varphi (y) \right| }
 >0,
$$
such that $U_s+\varepsilon \varphi \geq 0$ in $\Omega$ for
$0<\varepsilon<\varepsilon_0$. Now $U_s$ is a critical point of
$F(u)$ in $\Omega$ if and only if
\begin{align}
 \frac{d}{d \varepsilon} \Big |_{\varepsilon=0} F\left( U_s+ \varepsilon \varphi  \right)
 =0,\quad \forall \varphi \in C_0^\infty(\Omega).
\end{align}
The above definition yields that for any $\varphi \in C_0^\infty (\Omega)$ with $\int_\Omega \varphi(x)dx=0$
\begin{align}\label{support}
\int_{\Omega} \left( \log U_s - C_s \right) \varphi
 dx =0.
\end{align}
For any $\psi \in C_0^\infty(\Omega),$ denoting
\begin{align}
\varphi=\psi-\frac{U_s}{M_0} \int_{\Omega} \psi dx,
\end{align}
\er{support} becomes
\begin{align}
\int_{\Omega} \left(\mu_s-\int_{\Omega} \frac{U_s}{M_0} \mu_s dx\right) \psi dx=0,\quad \mbox{for~any}~~\psi \in C_0^\infty(\Omega).
\end{align}
This implies
\begin{align}\label{innercbar}
    \mu_s=\log U_s - C_s =\bar{C}, ~~ a.e. \mbox{ in }
    \Omega,
\end{align}
where
\begin{align}\label{wentai}
  \bar{C} =  \frac{1}{M_0}\int_{\Omega} U_s \log U_s -U_s C_s dx.
\end{align}
In addition, it follows from \er{logHLS} that
\begin{align}
2\int_{\R^2} U_s \log U_s dx+2M_0 (1+\log \pi-\log M_0) \ge \int_{\R^2} U_s C_s dx \ge \int_{\Omega} U_s C_s dx.
\end{align}
Recalling $U_s \log U_s \in L^1(\R^2)$ we can further deduce that
\begin{align}\label{cbar}
\bar{C} \ge \frac{1}{M_0} \left( \int_{\Omega}U_s\log U_s dx-2\int_{\R^2}U_s\log U_s dx-2 M_0 (1+\log \pi-\log M_0)\right).
\end{align}
Now we claim $\Omega=\R^2.$ If $\Omega$ is bounded, then \er{innercbar} implies $\bar{C}=-\infty$ at the boundary of $\Omega.$
This contradiction with \er{cbar} implies $\Omega$ is unbounded and the connected unbounded open set is $\R^2.$ Hence we complete the proof for (i) $\Leftrightarrow$ (iii) and (ii) $\Leftrightarrow$ (iii).\quad $\Box$
\end{proof}

From Proposition \ref{steadyproperties}, we can obtain that the stationary equation is
\begin{align}\label{uscs}
\left\{
  \begin{array}{ll}
    \log U_s-C_s=\bar{C},\quad \mbox{in}~~\R^2, \\
    -\Delta C_s=U_s, \quad \mbox{in}~~\R^2.
  \end{array}
\right.
\end{align}
Letting $\phi=\log U_s$ in \er{uscs}, the steady equation reduces to
\begin{align}\label{starstar}
-\Delta \phi=e^{\phi},\quad \mbox{in}~~\R^2.
\end{align}
It has been proved in \cite{wc91,wei83} that all the solutions $\phi(x)\in C^2(\R^2)$ of \er{starstar} uniquely assume the radial form up to translation
\begin{align}\label{phix}
\phi(x)=\log \frac{32 \lambda^2}{(4+\lambda^2 |x|^2)^2},~~\lambda>0.
\end{align}
It provides us with the explicit expression of the stationary solutions
\begin{align}
U_s(x)=e^{\phi(x)}=\frac{32 \lambda^2}{(4+\lambda^2 |x|^2)^2},~~\lambda>0,
\end{align}
with $\int_{\R^2} U_s(x)dx=8\pi.$ One readily check that there are nontrivial solutions to \er{2022} only in the case of a special value of the total mass. Combining with \er{M0} it gives our final estimates.
\begin{lemma}\label{M08pi}
There are nontrivial steady states to \er{fisherks} only for
\begin{align}
M_0=8\pi,
\end{align}
given by a one-parameter family with $\lambda>0$
\begin{align}\label{us}
U_s(x)=\frac{32 \lambda^2}{(4+\lambda^2 |x|^2)^2}
\end{align}
with infinite second moment $\int_{\R^2} |x|^2 U_s(x) dx=\infty.$
\end{lemma}

\subsection{Infinite time blow-up}
For $m_0<M_0=8\pi,$ combining \er{tritri} and Theorem \ref{globalexist} we deduce that solutions to system \er{rhorho} exist globally for any $0<t<\infty$ and satisfies
\begin{align}
\int_{\R^2} \rho(\cdot,t)|\log\rho(\cdot,t)|dx+\int_{\R^2}\rho(\cdot,t)|x|^2dx < \infty.
\end{align}
 In this subsection, we will further show that any solution of \er{fisherks} will converge towards a delta dirac at the center of mass via system \er{rhorho}. The main tool are the free energy functional
\begin{align}\label{estimates}
     F[\rho](t)+\int_0^t\left(\int_{\R^2} \rho|\nabla\left(\log \rho-m(s) w\right) |^2dx+\frac{m'(s)}{2}\int_{\R^2} \rho w dx\right)ds \le F[\rho_0].
\end{align}
with $F[\rho]:= \int_{\R^2}\rho \log \rho dx-\frac{m(t)}{2} \int_{\R^2}\rho w dx$. To prove this result, applying the mass conservation of $\rho$ we follow a procedure analogous to Lemma 3.6 and Corollary 3.2 in \cite{BCM08} to obtain
\begin{lemma}
Assume $\int_{\R^2} \left(|x|^2 +\left|\log \rho_0(x)\right|\right)\rho_0(x) dx<\infty$. If $\rho_0 \in L_+^1 \cap L^\infty(\R^2)$ and $m_0<M_0= 8\pi$, given  any solution $\rho$ to \er{rhorho}, we have
\begin{align}
    \lim_{t \to \infty}\int_{\R^2}\rho(x,t) \log \rho(x,t)dx = + \infty
\end{align}
and
\begin{align*}
    \lim_{t\to \infty} \rho(x,t) = \delta_{M_1}~~~weakly-^* ~as ~measures.
\end{align*}
where $M_1$ is the center of mass in infinite time.
\end{lemma}
\begin{proof}
Without loss of generality, let $\int_{R^2}x\rho(x,t)dx = 0$. Now we assume by  contradiction that there exists an increasing sequence of time $\{t_k\} \to \infty$ such that
\begin{align*}
    \int_{\R^2}\rho(x,t_k)\log \rho(x,t_k)dx < \infty.
\end{align*}
Combining with the boundedness of the second moment of $\rho$, we can find a subsequence of $\rho(x,t_k)$ (without relabelling for simplicity) that converges weakly to $\rho_\infty(x)$ in $L^1(\R^2)$ by Dunford-Pettis theorem. In addition, by $\rho=\frac{u(x,t)}{m(t)}$ and \er{m2t} one has that the second moment of the limiting density $\rho_\infty(x)$ satisfies
\begin{align}\label{rhoinfinte}
    0<\int_{\R^2} |x|^2 \rho_{\infty}(x)dx<\infty.
\end{align}

On the other hand, it follows from the definition \er{mt} of $m(t)$ and the logarithmic H-L-S inequality \er{logHLS} that $F(\rho)$ is bounded from below:
\begin{align*}
\begin{array}{ll}
      F(\rho)&=\int_{\R^2} \rho(x,t)\log \rho(x,t)dx-\frac{m(t)}{2}\int_{\R^2} \rho(x,t)w(x,t)dx  \\ \\
     & \ge \int_{\R^2} \rho(x,t)\log \rho(x,t)dx-4\pi\int_{\R^2} \rho(x,t)w(x,t)dx\\\\
     &\ge \int_{\R^2} \rho(x,t)\log \rho(x,t)dx + 2 \iint_{\R^2\times \R^2}\rho(x)\rho(y) \log|x-y| dx dy \\\\
     &\ge C,
\end{array}
\end{align*}
where the last line holds due to $\int_{\R^2} \rho(\cdot,t) dx=1$ for any $0<t<\infty$. So from \er{estimates} we have
\begin{align*}
    0 \le \lim_{t\to \infty}\int_0^{t}\left(\int_{\R^2}\rho|\nabla \log\rho-m(s)\nabla w|^2dx+\frac{m'(s)}{2}\int_{\R^2} \rho w dx\right)ds \le F[\rho_0]-\liminf_{t\to \infty}F[\rho](t).
\end{align*}
As a consequence,
\begin{align*}
\lim_{t\to \infty}\int_t^{\infty}\left(\int_{\R^2}\rho|\nabla \log\rho-m(s)\nabla w|^2dx+\frac{m'(s)}{2}\int_{\R^2}\rho w dx\right)ds=0.
\end{align*}
Noting that $\displaystyle \lim_{t\to\infty} m(t)=8\pi, \displaystyle \lim_{t\to\infty} m'(t)=0$,  we have, up to the extraction of subsequences , that the limit $\varrho_{\infty}(s,x)$ of $(s,x) \mapsto \varrho_{\infty}(t+s,x)$ when t goes to infinity satisfies
\begin{align}\label{pass to limit}
    \nabla\log \varrho_{\infty}-8\pi \nabla w_{\infty}=0,~
    w_{\infty}=-\frac{1}{2\pi} \log |\cdot|*\varrho_{\infty}.
\end{align}
Now reviewing Section \ref{31} we find that $\varrho_{\infty}$ are the family of stationary solutions to \er{rhorho} with infinite second momentum by virtue of the transformation \er{tritri} and Lemma \ref{M08pi}. This fact contradicts with \er{rhoinfinte} and thus the first result holds true.

Moreover, by making use of the mass conservation of $\rho$, a careful reading of the proof of Lemma 3.1 and Corollary 3.2 in \cite{BCM08} gives the latter desired result. \quad $\Box$
\end{proof}

Applying \er{tritri} we are able to show our main result of system \er{fisherks} in this section.
\begin{theorem}[Infinite Time Blow-up]
Assume $\int_{\R^2} \left( \left|\log u_0(x)\right|+|x|^2 \right)u_0(x) dx<\infty$ and $u_0 \in L_+^1 \cap L^\infty(\R^2)$. If $\int_{\R^2} u_0(x) dx <M_0 = 8\pi$,
the the weak solution $u$ to \er{fisherks} exists globally in time and converges to a delta dirac of mass $M_0$ concentrated at the center of mass as $t \to \infty$.
\end{theorem}

\section{Radially symmetric solutions}\label{sec5}
\def\theequation{4.\arabic{equation}}\makeatother
\setcounter{equation}{0}
\def\thetheorem{4.\arabic{theorem}}\makeatother
\setcounter{theorem}{0}

In this section, we firstly study the radially symmetric solutions of \er{rhorho}. The single equation we derive from \er{rhorho} with radial symmetry satisfies a comparison principle. We can reach the steady state of the reduced system and have a new intuition of the relationship among $M_0,m_0$ and $8\pi$ besides the second moment. Subsequently, we divide it into $M_0<8\pi$ and $ M_0>8\pi$. In these cases, the comparison principle plays an important role to capture the full structures of the reduced system. Finally, we transfer the system \er{rhorho} back to the original system \er{fisherks} and summarize the dynamical behaviors of radial symmetric solutions in different cases.

\subsection{Reduced system and steady states}

Radially symmetric solutions of the system \er{rhorho} on $\rho(t,r),w(t,r)$ is equivalent to
\begin{align}\label{rhorhoradial}
\left\{
      \begin{array}{ll}
      \frac{\partial}{\partial t} (r \rho)=(r \rho')'-m(t) (r\rho w')',~~& t \ge 0,~r > 0, \\
      -(rw')'=r \rho, ~~& t \ge 0,~r > 0, \\
      \rho'(t,r=0)=0, ~~& t \ge 0, \\
     \rho(t=0,r) =\rho_0(r) \ge 0, ~~& r \ge 0.
      \end{array}\right.
\end{align}
This can be reduced to a single equation on $M(t,r)$ which is defined as
\begin{align}
M(t,r)=2 \pi \int_0^r \sigma \rho(t,\sigma) d\sigma=-2\pi r w'(t,r).
\end{align}
Integrating \er{rhorhoradial} we arrive at
\begin{align}\label{Mradial}
\left\{
  \begin{array}{ll}
    \frac{\partial}{\partial t} M(t,r)=r \left( \frac{M'}{r} \right)'+\frac{m(t)}{2\pi r}M'M,~~& t \ge 0,~r>0, \\
     M(t,0)=0,~M(t,\infty)=1,~~& t \ge 0, \\
     M(0,r)=2\pi \int_0^r \sigma \rho_0(\sigma) d\sigma,~~& r \ge 0.
  \end{array}
\right.
\end{align}
Next we write for time independent solutions
\begin{align}
\left\{
  \begin{array}{ll}
   r^2 \left( \frac{\bar{M}'}{r} \right)'+\frac{M_0}{2\pi }\bar{M}'\bar{M}=r\bar{M}''-\bar{M}'+\frac{M_0}{4\pi }\left(\bar{M}^2\right)', ~~r>0,\\[1mm]
   \bar{M}(0)=0,~~\bar{M}'(r)>0.
  \end{array}
\right.
\end{align}
A simple manipulation leads to
\begin{align}\label{Msr}
\bar{M}_\lambda(r)=\frac{8\pi}{M_0\left(1+\lambda r^{-2}\right)}\quad \mbox{for some}~~\lambda>0.
\end{align}
Coming back to the density we obtain
\begin{align}\label{rhosr}
\bar{\rho}_\lambda(r)=\frac{8 \lambda}{M_0} \frac{1}{(r^2+\lambda)^2}.
\end{align}
It is worthwhile to note that \er{rhosr} is consistent with \er{us} up to a parameter. Recalling $\bar{M}(\infty)=1$, we need to emphasize that
there are radial steady states to \er{Mradial} only for $8\pi=M_0$, given by $\bar{M}_\lambda(r)=\frac{1}{1+\lambda r^{-2}}$.

\subsection{Refined existence and blow-up}

In terms of \er{rhorhoradial} and \er{Msr}, we can use a comparison argument to construct sub-or super-solutions in the spirit of \cite{bp07}. The analysis for the case $M_0>8\pi$ and $M_0<8\pi$ are quite different in view of the expression of \er{Msr}. Hence, in the following we shall separate these two cases to discuss.

\subsubsection{Global existence for $M_0<8\pi$}\label{radialglobal}
\begin{lemma}[Global solutions]\label{rhoglobal}
Let $m_0<M_0$ and $\bar{M}_{\lambda_0}(r)$ is defined as \er{Msr} for some $\lambda_0>0$. We assume that
\begin{align}
1<\frac{8\pi}{M_0}, \quad M(0,r)<\bar{M}_{\lambda_0}(r)
\end{align}
for any $r>0$, then the solution of \er{rhorhoradial} vanishes in $L^1(\R^2)$ locally. That's
\begin{align}
M(t,r) \to 0~~as~~t \to \infty
\end{align}
uniformly in interval $0\le r<R$ for any $R>0.$
\end{lemma}
\begin{proof}
Due to $M_0<8\pi$ and $M(0,r)<\bar{M}_{\lambda_0}(r)$, we may choose $0<\mu<1$ such that
\begin{align}\label{M0r}
\mu \frac{8\pi}{M_0}>1~~~\mbox{and}~~~ M(0,r)<\mu \bar{M}_{\lambda_0}(r),~\forall r>0.
\end{align}
In view of \er{mt} and $m_0<M_0$, we also have
\begin{align}\label{mtxiaoM0}
m_0<m(t)<M_0
\end{align}
for any $0< t<\infty$. Then we look for a supersolution to \er{rhorhoradial} as follows:
\begin{align}\label{superN}
\bar{N}(t,r)=\min\left( 1,\mu \frac{8\pi}{M_0\left(1+\lambda(t) r^{-2}\right)} \right)
\end{align}
where
\begin{align}\label{lamt}
\lambda(t)=\lambda_0+\frac{M_0}{\mu \pi}\left(1-\mu \right)\left( \mu \frac{8\pi}{M_0}-1 \right)t.
\end{align}
Furthermore, we may define $R(t)$ as
\begin{align}\label{Rt1}
1=\mu \frac{8\pi}{M_0\left(1+\lambda(t) R(t)^{-2}\right)}
\end{align}
which gives rise to
\begin{align}
\left\{
  \begin{array}{ll}
    r \ge R(t), & \bar{N}=1, \\
    r \le R(t), & \bar{N}=\mu \frac{8\pi}{M_0\left(1+\lambda(t) r^{-2}\right)}.
  \end{array}
\right.
\end{align}

Now we claim that $\bar{N}(t,r)$ is a supersolution to \er{rhorhoradial}. Denote the operator $L$
\begin{align}\label{L}
L(\bar{N}):=\frac{\partial}{\partial t} \bar{N}-r \left( \frac{\bar{N}'}{r} \right)'-\frac{m(t)}{2\pi r} \bar{N}'\bar{N}.
\end{align}
For $r \le R(t)$, with the fact \er{mtxiaoM0} we derive that
\begin{align}
L(\bar{N})=L\left(\mu \frac{8\pi}{M_0\left(1+\lambda(t) r^{-2}\right)}\right)&=\bar{N}\frac{r^{-2}}{1+\lambda(t)r^{-2}} \left( -\lambda'(t)+\frac{8\lambda(t) r^{-2}}{1+\lambda(t)r^{-2}} \left(1-\mu \frac{m(t)}{M_0}\right) \right) \nonumber \\
& \ge \bar{N}\frac{r^{-2}}{1+\lambda(t)r^{-2}} \left( -\lambda'(t)+\frac{8\lambda(t) R(t)^{-2}}{1+\lambda(t)R(t)^{-2}} \left(1-\mu \frac{m(t)}{M_0}\right) \right) \nonumber \\
&= \bar{N}\frac{r^{-2}}{1+\lambda(t)r^{-2}} \left( -\lambda'(t)+\frac{M_0}{\mu \pi}\left( 1-\mu \frac{m(t)}{M_0} \right) \left( \mu \frac{8\pi}{M_0}-1 \right) \right) \nonumber \\
&\ge \bar{N}\frac{r^{-2}}{1+\lambda(t)r^{-2}} \left( -\lambda'(t)+\frac{M_0}{\mu \pi}\left( 1-\mu \right) \left( \mu \frac{8\pi}{M_0}-1 \right) \right) \nonumber \\
&=0.
\end{align}
Here the last line holds by the choice \er{lamt} of $\lambda(t)$. We also find that for $r>R(t)$, $L(\bar{N})=L(1)=0$. Hence recalling $M(0,r)\le M(0,\infty)=1$ and \er{M0r} respectively, the comparison principle implies that
\begin{align}
M(t,r) \le \min\left( 1,\mu \frac{8\pi}{M_0\left(1+\lambda(t) r^{-2}\right)} \right)=\bar{N}(t,r)
\end{align}
for all $0 \le t<\infty$ and $0\le r<\infty.$ Moreover, by \er{lamt} we find
\begin{align}
\lambda(t) \to \infty,
\end{align}
then $R(t) \to \infty$ for $t$ large enough. Therefore, for any given interval $r \in (0,R)$ we have
\begin{align}
M(t,r) \le \mu \frac{8\pi}{M_0\left(1+\lambda(t) R^{-2}\right)} \to 0~~\mbox{as}~~t \to \infty.
\end{align}
Thus completes the proof. \quad $\Box$
\end{proof}

\subsubsection{Infinite time blow-up for $M_0>8\pi$}\label{radialblowup}
Next we prove a refined blow-up result (with a weaker condition than the second moment condition) which is close enough to exhibit the shape of the chemotactic collapse solution.

\begin{lemma}[Blow-up]\label{rhoblowup}
Assume $M_0<m_0$ and $\bar{M}_{\lambda_0}(r)$ is defined as \er{Msr} for some $\lambda_0>0$. If
\begin{align}
\frac{8\pi}{M_0}<1, \quad M(0,r)>\bar{M}_{\lambda_0}(r)
\end{align}
for any $r>0$ and there is no finite time blow-up, then the solution to \er{rhorhoradial} blows up in infinite time. Furthermore, the solution has a mass concentration at the origin, that's
\begin{align}
M(t,r(t)) >\frac{8\pi}{M_0}
\end{align}
for $r(t) \to 0$ as $t \to \infty$.
\end{lemma}
\begin{proof}
Since $\frac{8\pi}{M_0}<1$, We firstly choose $1<\mu_0<\mu_1$ such that
\begin{align}\label{mu0}
\mu_0 \frac{8\pi}{M_0}<\mu_1 \frac{8\pi}{M_0}<1~~~\mbox{and}~~~ M(0,r)>\mu_1 \bar{M}_{\lambda_0}(r)>\mu_0 \bar{M}_{\lambda_0}(r),~\forall r>0.
\end{align}
Combining $M_0<m_0$ with \er{mt} also leads to
\begin{align}\label{mtxiaom0}
M_0<m(t)<m_0.
\end{align}
We consider the function
\begin{align}
\underline{N}(t,r)=\max\left( \mu_1 \frac{8\pi}{M_0\left(1+\lambda_0 r^{-2}\right)},\mu_0 \frac{8\pi}{M_0\left(1+\lambda(t) r^{-2}\right)} \right),
\end{align}
where $\lambda(t)=\lambda_0 e^{A t}$ for some $A<0,$ we will prove that $\underline{N}$ is a subsolution as it is the maximum of two subsolutions.

Firstly we denote $N_1:=\mu_1 \frac{8\pi}{M_0\left(1+\lambda_0 r^{-2}\right)}$ and compute
\begin{align}
\frac{\partial}{\partial t} N_1-r \left( \frac{N_1'}{r} \right)'-\frac{m(t)}{2\pi r} N_1'N_1 &=N_1 \frac{r^{-2}}{1+\lambda_0 r^{-2}} \frac{8 \lambda_0}{\lambda_0+r^2} \left( 1-\mu_1 \frac{m(t)}{M_0} \right) \nonumber \\
& \le N_1 \frac{r^{-2}}{1+\lambda_0 r^{-2}} \frac{8 \lambda_0}{\lambda_0+r^2} \left( 1-\mu_1 \right) \nonumber \\
& <0.
\end{align}
Then it follows from \er{mu0} that $N_1$ is a subsolution to \er{rhorhoradial}. Secondly, we have to prove that $\mu_0 \frac{8\pi}{M_0\left(1+\lambda(t) r^{-2}\right)}$ is a subsolution in the interval $0<r\le R(t)$ where it achieves the maximum, i.e.
\begin{align}
\mu_1 \frac{8\pi}{M_0\left(1+\lambda_0 R(t)^{-2}\right)}=\mu_0 \frac{8\pi}{M_0\left(1+\lambda(t) R(t)^{-2}\right)}.
\end{align}
By the definition of $R(t)$, there exists a constant $R_0$ such that $r \le R(t) \le R_0$ with $\mu_1 \frac{8\pi}{M_0\left(1+\lambda_0 R_0^{-2}\right)}=\mu_0 \frac{8\pi}{M_0},$ then using \er{mtxiaom0} and a direct calculation provide
\begin{align}
\frac{\partial}{\partial t} \underline{N}-r \left( \frac{\underline{N}'}{r} \right)'-\frac{m(t)}{2\pi r} \underline{N}'\underline{N} &=
\underline{N} \frac{r^{-2}}{1+\lambda(t) r^{-2}} \left( -\lambda'(t)+\frac{8 \lambda(t)}{\lambda(t)+r^2} \left( 1-\mu_0 \frac{m(t)}{M_0}\right) \right)  \nonumber \\
& \le \underline{N} \frac{r^{-2}}{1+\lambda(t) r^{-2}} \left( -\lambda'(t)+\frac{8 \lambda(t)}{\lambda(t)+r^2} \left( 1-\mu_0 \right) \right) \nonumber \\
& \le \underline{N} \frac{r^{-2}}{1+\lambda(t) r^{-2}} \left( -\lambda'(t)+\frac{8 \lambda(t) R_0^{-2}}{\lambda_0 R_0^{-2}+1} \left( 1-\mu_0 \right) \right) \nonumber \\
& =\underline{N} \frac{r^{-2}}{1+\lambda(t) r^{-2}} \left( -\lambda'(t)+\frac{8 \mu_0 R_0^{-2}\lambda(t)}{\mu_1} \left( 1-\mu_0 \right) \right) \nonumber \\
& =0
\end{align}
by choosing
\begin{align}
\lambda(t)=\lambda_0 e^{-\frac{8 \mu_0 R_0^{-2}}{\mu_1}(\mu_0-1) t}=\lambda_0 e^{A t}.
\end{align}
We notice that $\underline{N}=\mu_1 \frac{8\pi}{M_0\left(1+\lambda_0 r^{-2}\right)}$ for $r>R(t)$. Now keeping \er{mu0} in mind, the comparison principle gives
\begin{align}
M(t,r) \ge \underline{N}(t,r) \ge \mu_0 \frac{8\pi}{M_0\left(1+\lambda(t) r^{-2}\right)}=\mu_0 \frac{8\pi}{M_0\left(1+\lambda_0 e^{A t} r^{-2}\right)}.
\end{align}
Therefore, at $r=e^{\frac{A}{3}t}$ we find
\begin{align}
M(t,r) \ge \mu_0 \frac{8\pi}{M_0\left(1+\lambda_0 e^{\frac{A}{3} t}\right)}
\end{align}
and the subsequent mass concentration at $r=0$
\begin{align}
M(t,r(t)) > \frac{8\pi}{M_0}
\end{align}
for $r(t) \to 0$ as $t \to \infty.$ This is our desired estimate. \quad $\Box$
\end{proof}

\subsection{Main results for the radially symmetric solutions of \er{fisherks}}
Now we are ready to summarize the main results of \er{rhorhoradial} as follows.

\begin{theorem}\label{rhothm}
Assume that the initial data $\rho(0,r)$ is radially symmetric, $\bar{\rho}_{\lambda_0}(r)$ is defined as \er{rhosr} for some $\lambda_0>0.$
\begin{itemize}
  \item[(1)] If $M_0<8\pi$ and
   \begin{align}
     \rho(0,r)<\bar{\rho}_{\lambda_0}(r)
   \end{align}
  for any $r>0,$ then the radially symmetric solution of \er{rhorhoradial} vanishes in $L^1(\R^2)$ locally as $t \to \infty.$
  \item[(2)] If $M_0>8\pi$ and
   \begin{align}
     \rho(0,r)>\bar{\rho}_{\lambda_0}(r)
   \end{align}
  for any $r>0,$ then there could exist one solution to \er{rhorhoradial} which blows up in finite time or has a mass concentration at the origin such that
  \begin{align}
    2\pi \int_0^{r(t)} \rho(t,\sigma) \sigma d\sigma > \frac{8\pi}{M_0}
  \end{align}
  for $r(t) \to 0$ as $t \to \infty.$
\end{itemize}
\end{theorem}

By using \er{tritri}, we transfer Theorem \ref{rhothm} to the original system \er{fisherks}.

\begin{theorem}\label{uthm}
Assume that $u_0(r)$ is the radially symmetric initial data to \er{fisherks}.
\begin{itemize}
  \item[(1)] If $m_0<M_0<8\pi$ and
   \begin{align}
     u_0(r)<\frac{8 m_0}{M_0}\frac{\lambda_0}{\left(\lambda_0+r^2\right)^2},~~\forall r>0
   \end{align}
  for some $\lambda_0>0,$ then any radially symmetric solution of \er{fisherks} vanishes in $L^1(\R^2)$ locally as $t \to \infty.$
  \item[(2)] If $m_0>M_0>8\pi$ and
   \begin{align}
     u_0(r)>\frac{8 m_0}{M_0}\frac{\lambda_0}{\left(\lambda_0+r^2\right)^2},~~\forall r>0
   \end{align}
  for some $\lambda_0>0,$ then any radially symmetric solution of \er{fisherks} has finite time blow-up or has a mass concentration at the origin such that
  \begin{align}
    2\pi \int_0^{r(t)} u(t,\sigma) \sigma d\sigma > 8\pi
  \end{align}
  for $r(t) \to 0$ as $t \to \infty.$
\end{itemize}
\end{theorem}

we comment that, in the blow-up result, the limitation on the second moment are in fact useless in the radially symmetric case and they can be relaxed by the comparison condition with the steady state \er{Msr} which corresponds to an infinite second moment. In contrast to Section \ref{sec2}, infinite initial second moment can produce infinite time blow-up.

\begin{appendix}
\section*{Appendix}
\begin{lemma}\label{ulogusemi}
Let $u_\varepsilon,u \in L_+^1(\R^2)$ satisfy
\begin{align}
&\int_{\R^2} |x|^2 u_\varepsilon dx<\infty,\quad \int_{\R^2} |x|^2 u dx<\infty, \\
&\int_{\R^2} u_\varepsilon |\log u_\varepsilon| dx <\infty, \\
&u_\varepsilon \rightharpoonup u ~~(\varepsilon \to 0)~~ \mbox{weakly~~in~~} L^1(\R^2).
\end{align}
Then
\begin{align}
\int_{\R^2} u(x) |\log u(x)| dx<\infty,
\end{align}
and
\begin{align}\label{star11}
\int_{\R^2} u(x) \log u(x) dx \le \displaystyle \liminf_{\varepsilon \to 0}\int_{\R^2} u_\varepsilon(x) \log u_\varepsilon(x) dx.
\end{align}
\end{lemma}
\begin{proof}
In the following, $C,C_1,C_2$ stand for finite positive constants that are independent of $\varepsilon$. We will give the proof step by step.
\begin{itemize}
  \item[\textbf{(i)}] We firstly show that
   \begin{align}\label{3}
    \displaystyle \lim_{\varepsilon \to 0} \int_{\R^2} (1+|x|^\beta) u_\varepsilon(x) dx=\int_{\R^2} (1+|x|^\beta) u(x) dx,\quad \mbox{for~~any}~~0<\beta<2.
   \end{align}
\end{itemize}
In fact, for any $R>1,$
\begin{align*}
 & \left| \int_{\R^2}(1+|x|^\beta)(u_\varepsilon(x)-u(x))dx   \right| \\
\le & \left| \int_{\R^2} (1+|x|^\beta) 1_{\{|x|\le R \}} (u_\varepsilon(x)-u(x))dx+\int_{|x|>R}(1+|x|^\beta) u_\varepsilon(x)dx+ \int_{|x|>R}(1+|x|^\beta) u(x)dx \right|.
\end{align*}
We deal with the last two terms as follows.
\begin{align*}
&\int_{|x|>R}(1+|x|^\beta) u_\varepsilon(x)dx+ \int_{|x|>R}(1+|x|^\beta) u(x)dx \\
\le & 2 \left( \int_{|x|>R} |x|^\beta u_\varepsilon(x) dx+ \int_{|x|>R} |x|^\beta u(x) dx   \right) \\
\le & \frac{2}{R^{2-\beta}} \left( \int_{\R^2} |x|^2 u_\varepsilon(x) dx+ \int_{\R^2} |x|^2 u(x) dx   \right) \le \frac{C}{R^{2-\beta}}.
\end{align*}
Since, by weak convergence,
\begin{align*}
\left| \int_{\R^2} (1+|x|^\beta) 1_{\{|x|\le R \}} (u_\varepsilon(x)-u(x))dx \right| \to 0,\quad \varepsilon \to 0.
\end{align*}
It follows that
\begin{align*}
\displaystyle \limsup_{\varepsilon \to 0} \left| (1+|x|^\beta)(u_\varepsilon(x)-u(x))\right| \le \frac{C}{R^{2-\beta}},\quad \forall R>1.
\end{align*}
Letting $R \to \infty$ leads to \er{3}.
\begin{itemize}
  \item[\textbf{(ii)}] Let's prove
   \begin{align}\label{1}
     \int_{\R^2} u(x) |\log u(x)| dx<\infty.
   \end{align}
\end{itemize}
Actually, by convexity of $y \mapsto y \log y, y \ge 0,$ one has
\begin{align*}
u_\varepsilon \log u_\varepsilon  \ge u \log u+(1+\log u) (u_\varepsilon-u).
\end{align*}
From this we have, for any $m \in \mathbb{N}$,
\begin{align*}
& 1_{\{ 1\le u(x)\le m \}}u_\varepsilon \log u_\varepsilon  \\
\ge & 1_{\{ 1\le u(x)\le m \}}u \log u+1_{\{ 1\le u(x)\le m \}}(1+\log u) (u_\varepsilon-u).
\end{align*}
Hence
\begin{align*}
&\int_{\R^2} 1_{\{ 1\le u(x)\le m \}}u \log u dx+\int_{\R^2} 1_{\{ 1\le u(x)\le m \}}(1+\log u) (u_\varepsilon-u) dx \\
&\le \int_{\R^2} 1_{\{ 1\le u(x)\le m \}}u_\varepsilon \log u_\varepsilon dx \\
&\le \displaystyle \sup_{\varepsilon >0} \int_{\R^2} u_\varepsilon |\log u_\varepsilon| dx=:C_1 (<\infty).
\end{align*}
Since, by weak convergence,
\begin{align*}
\displaystyle \lim_{\varepsilon \to 0} \int_{\R^2} 1_{\{ 1 \le u(x) \le m \}} (1+\log u) (u_\varepsilon-u) dx=0,
\end{align*}
it follows by passing to the limit $\varepsilon \to 0$ that
\begin{align*}
\int_{\R^2} 1_{\{ 1\le u(x)\le m \}}u \log u dx \le C_1.
\end{align*}
Then letting $m \to \infty$ gives by Levi monotone convergence
\begin{align*}
\int_{\R^2} 1_{\{ u(x) \ge 1 \}}u \log u dx \le C_1.
\end{align*}
This together with \er{xiaoyu1} and the boundedness of $\int_{\R^2} |x|^2 u(x)dx $ yields
\begin{align*}
\int_{\R^2} u(x) |\log u(x)| dx &=\int_{\R^2} 1_{\{ u(x)<1\} } u(x) |\log u(x)| dx+ \int_{\R^2} 1_{\{ u(x) \ge 1 \}} u(x)|\log u(x)| dx \nonumber \\
& \le \int_{\R^2} e^{-|x|^2} dx+2\int_{\R^2} u(x) |x|^2 dx+C_1<\infty.
\end{align*}
Hence we close the proof of \er{1}.
\begin{itemize}
  \item[\textbf{(iii)}] Next, to prove \er{star11}, we consider the truncation
      \begin{align*}
         g_m(x)=u(x) \wedge m+\frac{1}{m} e^{-|x|}.
      \end{align*}
By convexity we have
\begin{align*}
u_\varepsilon \log u_\varepsilon  \ge g_m(x)\log g_m(x)+(1+\log g_m(x))(u_\varepsilon-g_m),
\end{align*}
i.e.
\begin{align}\label{star22}
g_m(x)\log g_m(x) \le u_\varepsilon \log u_\varepsilon +(1+\log g_m(x))(g_m -u_\varepsilon).
\end{align}
\end{itemize}

We firstly use \er{1} to check the condition of $g_m \log g_m$ for using Lebesgue dominated convergence, i.e.
\begin{align}\label{2}
\int_{\R^2} g_m |\log g_m|dx \le C_2.
\end{align}
By convexity of $y \mapsto y \log y$ on $y \ge 0$ one has
\begin{align}\label{6}
(a+b)\log (a+b) &=\left( \frac{1}{2} 2a+\frac{1}{2} 2b  \right) \log \left( \frac{1}{2} 2a+\frac{1}{2} 2b  \right)  \nonumber \\
&\le \frac{1}{2} 2a \log 2a+\frac{1}{2} 2b \log 2b,~~a,b \ge 0.
\end{align}
We also have
\begin{align}\label{7}
(a+b) \log \frac{1}{a+b} \le a \log \frac{1}{a}+b \log \frac{1}{b},~~a,b \ge 0.
\end{align}
Now we use \er{6} to see that if $g_m \ge 1,$ then, because $m \ge 4,$ we must have $u(x)>1/2,$ so that
\begin{align*}
&g_m(x) |\log g_m(x)|=g_m(x) \log g_m(x) \\
\le & (u(x)\wedge m) \log (2 u(x)\wedge m)+\frac{1}{m}e^{-|x|}\log \left( \frac{2}{m} e^{-|x|} \right) \\
\le & u(x) \log (2 u(x))+\frac{1}{m}e^{-|x|} \frac{2}{m}  e^{-|x|} \\
\le & u(x)+u(x) |\log u(x)|+e^{-2|x|}.
\end{align*}
While if $g_m(x)<1,$ then $u(x)<1,$ then by \er{7} we have
\begin{align*}
& g_m(x) |\log g_m(x)|=\left(  u(x)+\frac{1}{m}e^{-|x|} \right) \log \left( \frac{1}{u(x)+\frac{1}{m}e^{-|x|}}   \right) \\
\le & u(x)|\log u(x)|+\frac{1}{m} e^{-|x|} \log (m e^{|x|}) \\
\le & u(x) |\log u(x)|+\frac{2}{\sqrt{m}} e^{-|x|/2},
\end{align*}
where we used $\log y=2\log \sqrt{y} \le 2 \sqrt{y}$ for $y \ge 0.$ So we have proved an $L^1$ control
\begin{align}\label{8}
g_m(x) |\log g_m(x)| \le u(x)+u(x)|\log u(x)|+e^{-|x|/2}+e^{-2|x|}.
\end{align}

The aim next is to deal with the second term in the right hand side of \er{star22}. Using
\begin{align*}
u(x) \wedge m=u(x)-(u(x)-m)_+,
\end{align*}
we handle
\begin{align}\label{9}
&(1+\log g_m(x))(g_m(x)-u_\varepsilon(x)) \nonumber \\
=& (1+\log g_m(x))(u(x)-u_\varepsilon(x))-(1+\log g_m(x))(u(x)-m)_++(1+\log g_m(x))\frac{1}{m}e^{-|x|}.
\end{align}
By considering $g_m(x)\ge 1$ and $g_m(x)<1$ we have
\begin{align*}
|1+\log g_m(x)| \le 1+\left( \log (m+1) \vee (|x|+\log m) \right) \le 2(1+|x|)\log (m+1).
\end{align*}
Therefore we have
\begin{align}\label{10}
|1+\log g_m(x)| \frac{1}{m}e^{-|x|} \le 2 \frac{\log (m+1)}{m} (1+|x|)e^{-|x|},
\end{align}
and by \er{3},
\begin{align}\label{110}
\displaystyle \lim_{\varepsilon \to 0} \int_{\R^2} (1+\log g_m(x)) (u(x)-u_\varepsilon(x))=0.
\end{align}
Moreover we have
\begin{align}\label{12}
-(1+\log g_m(x)) (u(x)-m)_+ \le 0.
\end{align}
In fact the case $u(x) \le m$ is obvious. For the case $u(x)>m$ we have $g_m(x)>m>1,$ so that $\log g_m(x)>0.$ So \er{12} holds true.

From \er{star22}, collecting \er{9}-\er{12} we have by taking lower limit $\varepsilon \to 0$ that
\begin{align}\label{13}
\int_{\R^2} g_m(x) \log g_m(x) dx \le \displaystyle \liminf_{\varepsilon \to 0} \int_{\R^2} u_\varepsilon(x) \log u_\varepsilon(x) dx+2\frac{\log (m+1)}{m} \int_{\R^2} (1+|x|)e^{-|x|} dx,\quad \forall m \ge 4.
\end{align}
Finally by continuity one has
\begin{align}
g_m(x) \log g_m(x) \to u(x) \log u(x) \quad (m \to \infty).
\end{align}
Then it follows \er{8} and Lebesgue dominated convergence and \er{13} that
\begin{align}
\int_{\R^2} u(x) \log u(x) dx=\displaystyle \lim_{m \to \infty} \int_{\R^2} g_m(x) \log g_m(x) dx \le \displaystyle \liminf_{\varepsilon \to 0} \int_{\R^2} u_\varepsilon(x) \log u_\varepsilon(x) dx.
\end{align}
Thus we close the proof. \quad $\Box$
\end{proof}
\end{appendix}

\end{document}